\documentclass[10pt,a4paper]{article}
\usepackage[utf8]{inputenc}
\usepackage[a4paper, left=3cm, right=2cm]{geometry}
\usepackage[T1]{fontenc}
\usepackage{amsmath}
\usepackage{amsthm}
\usepackage{amsfonts}
\usepackage{amssymb}
\usepackage{bbm}
\usepackage{graphicx}
\usepackage{enumerate}
\usepackage{listings}
\usepackage{color}
\usepackage{mathrsfs}
\usepackage{pdfpages}
\usepackage{cite}
\usepackage{tikz}
\usepackage{xcolor}
\usepackage{hyperref}
\usepackage{authblk}
\hypersetup{
	colorlinks,
	linkcolor=black,
	citecolor=green,
	urlcolor=blue} 
\definecolor{mygreen}{RGB}{28,172,0} 
\definecolor{mylilas}{RGB}{170,55,241}
\author[1]{Petru A. Cioica-Licht}
\author[2]{Martin Hutzenthaler}
\author[3]{P. Tobias Werner}

\affil[1]{Institute of Mathematics, University of Kassel, \texttt{cioica-licht@mathematik.uni-kassel.de}}
\affil[2]{Faculty of Mathematics, University of Duisburg-Essen, \texttt{martin.hutzenthaler@uni-due.de}}
\affil[3]{ Institute of Mathematics, University of Kassel,\texttt{twerner@mathematik.uni-kassel.de}}
\title{Deep neural networks overcome the curse of dimensionality in the numerical approximation of semilinear partial differential equations}

\theoremstyle{plain}
\newtheorem{theorem}{Theorem}[section]

\theoremstyle{theorem}

\newtheorem{corollary}[theorem]{Corollary}
\newtheorem{lemma}[theorem]{Lemma}

\newtheorem{setting}[theorem]{Setting}

\newcommand{\R}{\mathbb{R}}
\newcommand{\F}{\mathbb{F}}
\newcommand{\E}{\mathbb{E}}
\newcommand{\N}{\mathbb{N}}
\newcommand{\calR}{\mathcal{R}}
\newcommand{\calP}{\mathcal{P}}
\newcommand{\calD}{\mathcal{D}}
\newcommand{\calY}{\mathcal{Y}}
\newcommand{\calX}{\mathcal{X}}
\newcommand{\calV}{\mathcal{V}}
\newcommand{\frakg}{\mathfrak{g}}
\newcommand{\frakf}{\mathfrak{f}}

\newcommand{\frakt}{\mathfrak{t}}
\newcommand{\frakT}{\mathfrak{T}}

\newcommand{\frakv}{\mathfrak{v}}
\newcommand{\tmu}{\tilde{\mu}}
\newcommand{\tsigma}{\tilde{\sigma}}
\newcommand{\nrm}[1]{\left\lVert #1 \right\rVert}
\newcommand{\mnrm}[1]{{\left\vert\kern-0.25ex\left\vert\kern-0.25ex\left\vert #1 
		\right\vert\kern-0.25ex\right\vert\kern-0.25ex\right\vert}}
\begin{document}
	\pagenumbering{Roman}
	\maketitle
	\begin{abstract}
		\noindent We prove that deep neural networks are capable of approximating solutions of semilinear Kolmogorov PDE in the case of gradient-independent, Lipschitz-continuous nonlinearities, while the required number of parameters in the networks grow at most polynomially in both dimension $ d\in\N $ and prescribed reciprocal accuracy $ \varepsilon $. Previously, this has only been proven in the case of semilinear heat equations. 
	\end{abstract}
	\tableofcontents
	\hypersetup{
		colorlinks,
		linkcolor=red,
		citecolor=violet,
		urlcolor=blue} 
	\pagenumbering{arabic}
	\section{Introduction}
	Deep-learning based algorithms are employed in a wide field of real world applications and are nowadays the standard approach for most machine learning  related problems. They are used extensively in face and speech recognition, fraud detection, function approximation, and solving of partial differential equations (PDEs). The latter are utilized to model numerous phenomena in nature, medicine, economics, and physics. Often, these PDEs are nonlinear and high-dimensional. For instance, in the famous Black-Scholes model the PDE dimension $ d\in\N $ corresponds to the number of stocks considered in the model. Relaxing the rather unrealistic assumptions made in the model results in a loss of linearity within the corresponding Black-Scholes PDE, cf., e.g.,    \cite{ANKUDINOVA2008799,Bergman_Yaacov_BS_interest_rates}.
	These PDEs, however, can not in general be solved explicitly and therefore need to be approximated numerically.
	Classical grid based approaches such as finite difference or finite element methods  suffer from the \textit{curse of dimensionality} in the sense that the computational effort grows exponentially in the PDE dimension $ d\in \N $ and are therefore not appropriate.
	Among the most prominent and promising approaches to approximate high-dimensional PDEs are deep-learning algorithms, which seem to handle high-dimensional problems well, cf., e.g., \cite{Han8505,beck2021deep,Beck_2019,E_2017,beck2021overview}. 
	However, compared to the vast field of applications these techniques are successfully applied to  nowadays, there exist only few theoretical results proving that deep neural networks do not suffer from the curse of dimensionality in the numerical approximation of partial differential equations, see \cite{Hutzenthaler_2020,jentzen2019proof,Berner_2020,grohs2018proof,grohs2020deep,Reisinger2020deep,Takahashi_Yamada2021deep,grohs2021deep}.\\
	The key contribution of this article is a rigorous proof that deep neural networks do overcome the curse of dimensionality in the numerical approximation of semilinear partial differential equations, where the nonlinearity, the coefficients determining the linear part, and the initial condition of the PDEs are Lipschitz continuous. In particular, Theorem \ref{INTROTHM} below proves that the number of parameters in the approximating deep neural network grows at most polynomially in both the reciprocal of the described approximation accuracy $ \varepsilon \in (0,1]$ and the PDE dimension $ d\in\N $:
	\begin{theorem}\label{INTROTHM}
		Let $ \nrm{\cdot}\colon(\cup_{d \in \N}\R^d)\rightarrow[0,\infty) $ and $ \mathbf{A}_d\colon\R^d\rightarrow\R^d, d\in\N, $ satisfy for all $ d\in\N, x=(x_1,...,x_d)\in\R^d $ that $ \nrm{x} = \sqrt{\sum_{i=1}^{d}(x_i)^2}$ and $ \mathbf{A}_d(x)=(\max\{x_1,0\},...,\max\{x_d,0\}) $,
		let $ \nrm{\cdot}_F\colon(\cup_{d\in\N}\R^{d\times d})$  $\rightarrow [0,\infty) $ satisfy for all $d\in\N, A= (a_{i,j})_{i,j\in\{1,...,d\}}\in\R^{d\times d} $  that $ \nrm{A}_F=  \sqrt{\sum_{i,j=1}^{d}(a_{i,j})^2}$,
		let $ \langle \cdot,\cdot\rangle\colon$  $(\cup_{d \in \N}\R^d$  $\times\R^d)\rightarrow\R $ satisfy for all $ d\in \N, $ $ x=(x_1,...,x_d), $  $y=(y_1,...,y_d)\in\R^d $ that $ \langle x,y\rangle = \sum_{i=1}^d x_iy_i, $
		let $ \mathbf{N}=\cup_{H\in\N}\cup_{(k_0,...,k_{H+1})\in\N^{H+2}}[\prod_{n=1}^{H+1}(\R^{k_n\times k_{n-1}}\times \R^{k_n})] $, let $ \calR\colon\mathbf{N}\rightarrow (\cup_{k,l\in\N} C(\R^k,\R^l)), $
		$ \calD\colon\mathbf{N}\rightarrow \cup_{H\in\N}\N^{H+2}, $
		and $ \mathcal{P}\colon\mathbf{N}\rightarrow\N  $ satisfy for all $ H\in\N, k_0,k_1,...,k_H,k_{H+1}\in \N,$  $ \Phi=((W_1,B_1),...,(W_{H+1},B_{H+1}))\in\prod_{n=1}^{H+1}(\R^{k_n\times k_{n-1}}\times\R^{k_n}), x_0\in \R^{k_0},...,x_H\in\R^{k_{H}} $  with $ \forall n\in \N\cap[1,H]\colon x_n = \mathbf{A}_{k_n}(W_nx_{n-1}+B_n)  $ that 
		\begin{equation}
			\mathcal{R}(\Phi)\in C(\R^{k_0},\R^{k_{H+1}}),\quad (\mathcal{R}(\Phi))(x_0)=W_{H+1}x_H+B_{H+1},
		\end{equation}
		\begin{equation}
			\calD(\Phi)=(k_0,k_1,...,k_H,k_{H+1}),  \text{ and } \mathcal{P}(\Phi)=\sum_{n=1}^{H+1}k_n(k_{n-1}+1),
		\end{equation}
		let $ c,T\in(0,\infty) $, $ f\in C(\R,\R) $, for all $ d\in\N $ let $ g_d \in C(\R^d,\R) $, $ u_d \in C^{1,2}([0,T]\times\R^d,\R) $, $ \mu_d=(\mu_{d,i})_{i\in\{1,...,d\}}\in C(\R^d,\R^d) $, $ \sigma_d = (\sigma_{d,i,j})_{i,j\in\{1,...,d\}}\in C(\R^d,\R^{d\times d}) $ satisfy for all $ t\in[0,T], x=(x_1,...,x_d),$  $y =(y_1,...,y_d)\in \R^d $ that $ u_d(0,x)=g_d(x), $
		\begin{equation}
			|f(x_1)-f(y_1)|\le c|x_1-y_1|, \quad |u_d(t,x)|^2\le c\left[d^c+\nrm{x}^2\right],
		\end{equation}
		\begin{equation}\label{INTROTHM_PDE}
			\dfrac{\partial }{\partial t} u_d (t,x) =\langle(\nabla_x u_d)(t,x), \mu_d(x)\rangle+\tfrac{1}{2}\textup{Tr}\left(\sigma_d(x)[\sigma_d(x)]^*(\textup{Hess}_xu_d)(t,x)\right)+f(u_d(t,x)),
		\end{equation}
		for all $ d\in\N,v\in\R^d, \varepsilon\in(0,1] $ let $ \hat{\sigma}_{d,\varepsilon}\in C(\R^d,\R^{d\times d}), $ $ \frakg_{d,\varepsilon}, \tmu_{d,\varepsilon},\tsigma_{d,\varepsilon,v}\in\mathbf{N}  $ satisfy for all $d\in\N, x,y,v\in\R^d, \varepsilon\in(0,1], $  
		that $ \calR(\frakg_{d,\varepsilon})\in C(\R^d,\R) $, $\calR(\tmu_{d,\varepsilon})\in C(\R^d,\R^d)  $, 
		$ \calR(\tsigma_{d,\varepsilon,v})\in C(\R^d,\R^d),$  $ (\calR(\tsigma_{d,\varepsilon,v}))(x)=\hat{\sigma}_{d,\varepsilon}(x)v,$  $ \calD(\tsigma_{d,\varepsilon,v})=\calD(\tsigma_{d,\varepsilon,0}),$
		\begin{equation}
			|(\calR(\frakg_{d,\varepsilon}))(x)|^2+\max_{i,j\in\{1,...,d\}}\left(\left |[(\calR(\tmu_{d,\varepsilon}))(0)]_i\right|+|(\hat{\sigma}_{d,\varepsilon}(0))_{i,j}|\right)\le c\left[d^c+\nrm{x}^2\right],
		\end{equation}
		\begin{equation}
			\max\{|g_d(x)-(\calR(\frakg_{d,\varepsilon}))(x)|, \nrm{\mu_d(x)-(\calR(\tmu_{d,\varepsilon}))(x)}, \nrm{\sigma_d(x)-\hat{\sigma}_{d,\varepsilon}(x)}_F\}\le \varepsilon cd^c(1+\nrm{x}^c),
		\end{equation}
		\begin{equation}
			\max\{|(\calR(\frakg_{d,\varepsilon}))(x)-(\calR(\frakg_{d,\varepsilon}))(y)|,\nrm{(\calR(\tmu_{d,\varepsilon}))(x)-(\calR(\tmu_{d,\varepsilon}))(y)},\nrm{\hat{\sigma}_{d,\varepsilon}(x)-\hat{\sigma}_{d,\varepsilon}(y)}_F \}\le c\nrm{x-y},
		\end{equation}
		\begin{equation}
			\max\{\calP(\frakg_{d,\varepsilon}),\calP(\tmu_{d,\varepsilon}),\calP(\tsigma_{d,\varepsilon,\textcolor{blue}{v}})\}\le cd^c\varepsilon^{-c}.
		\end{equation}
		Then there exists $ \eta \in (0,\infty),  $ such that for all $ d\in\N, \varepsilon\in(0,1] $ there exists a $ \Psi_{d,\varepsilon}\in\mathbf{N} $ satisfying $ \calR(\Psi_{d,\varepsilon})\in C(\R^d,\R), $ $ \calP(\Psi_{d,\varepsilon})\le \eta d^\eta \varepsilon^{-\eta} $ and \\
		\begin{equation}
			\left[\int_{[0,1]^d}|u_d(T,x)-(\calR(\Psi_{d,\varepsilon}))(x)|^2dx\right]^{\frac{1}{2}}\le \varepsilon.
		\end{equation}
	\end{theorem}
	
	\noindent
	Theorem \ref{INTROTHM} follows from  Corollary \ref{MainCorollary}  (applied with $  u_d(t,x)\curvearrowleft u_d(T-t,x) $ for $ t\in [0,T], x\in\R^d $ in the notation of Corollary \ref{MainCorollary}), which in turn is a special case of Theorem \ref{MainTheorem}, the main result of this article.\\
	Let us comment on the mathematical objects appearing in Theorem \ref{INTROTHM}. By $ \mathbf{N} $ in Theorem \ref{INTROTHM} above we denote the set of all artificial feed-forward deep neural networks (DNNs), by $ \calR $ in Theorem \ref{INTROTHM} above we denote the operator that maps each DNN to its corresponding function, by $ \calP $ in Theorem \ref{INTROTHM} above we denote the function that maps a DNN to its number of parameters. 
	For all $ d\in \N $ we have that $ \mathbf{A}_d\colon\R^d\rightarrow\R^d$ in Theorem \ref{INTROTHM} above refers to the componentwise applied rectified linear unit (ReLU) activation fucntion. 
	The functions $ \mu_d\colon\R^d\rightarrow\R^d $ and $ \sigma_d\colon\R^d\rightarrow\R^{d\times d} $ in Theorem \ref{INTROTHM} above specify the linear part of the PDEs in (\ref{INTROTHM_PDE}). The functions $ g_d\colon\R^d\rightarrow\R,d\in\N, $ describe the initial condition, while the function $ f\colon\R\rightarrow\R $ specifies the nonlinearity of the PDEs in (\ref{INTROTHM_PDE}) above.
	The real number  $ T\in(0,\infty) $ describes the  time horizon of the PDEs in (\ref{INTROTHM_PDE}) in Theorem \ref{INTROTHM} above.
	The real number $ c\in (0,\infty) $ in Theorem \ref{INTROTHM} above is used to formulate a growth condition and a Lipschitz continuity condition of the functions $u_d\colon\R^d\rightarrow\R, \calR(\tmu_{d,\varepsilon})\colon\R^d\rightarrow\R^d, \hat{\sigma}_{d,\varepsilon}\colon\R^d\rightarrow\R^{d\times d} $,  $\calR( \frakg_{d,\varepsilon})\colon\R^d\rightarrow\R,d\in\N, \varepsilon\in(0,1] $, and $ f\colon \R\rightarrow\R.  $ 
	Furthermore, the real number $ c\in(0,\infty) $ in Theorem \ref{INTROTHM} above formulates an approximation condition between the functions $ \mu_d,\sigma_d, g_d $ and $  \calR(\tmu_{d,\varepsilon}), \hat{\sigma}_{d,\varepsilon},\calR(\frakg_{d,\varepsilon}), d\in\N, \varepsilon\in(0,1] $, as well as an upper boundary condition for the number of parameters involved in the DNNs $ \frakg_{d,\varepsilon},\tmu_{d,\varepsilon}, \tsigma_{d,\varepsilon,v}, d\in\N,v\in\R^d, \varepsilon\in(0,1] $.
	These assumptions guarantee unique existence of  \textit{viscosity solutions} of the PDEs in (\ref{INTROTHM_PDE}) in Theorem \ref{INTROTHM} above (see Theorem \ref{MainTheorem} and cf., e.g., \cite{users_guide_viscosity_solutions} for the notion of viscosity solution), and that the functions $ \mu_d, \sigma_d, g_d ,d\in\N,$ can be approximated without the curse of dimensionality by means of DNNs.
	The latter is the key assumption in Theorem \ref{INTROTHM} above and means in particular, that Theorem \ref{INTROTHM} can be, very roughly speaking, reduced to the statement that if DNNs are capable of approximating the initial condition and the linear part of the PDEs in (\ref{INTROTHM_PDE}) without the curse of dimensionality, then they are also capable of approximating its solution without the curse of dimensionality.\\
	Note that the statement of Theorem \ref{INTROTHM} is purely deterministic. The proof of Theorem \ref{MainTheorem}, however,  heavily relies on probabilistic tools.
	It is influenced by the method employed in \cite{Hutzenthaler_2020} and is built on the theory of full history recursive multilevel Picard approximations (MLP), which are numerical approximation schemes that have been proven to overcome the curse of dimensionality in numerous settings, cf., e.g., \cite{e2019multilevel,E_2019,hutzenthaler2020multilevel,beck2021history,Hutzenthaler_2020_cod_semilinearPDE}.
	In particular, in \cite{hutzenthaler2020multilevel} a MLP approximation scheme is introduced that overcomes the curse of dimensionality in the numerical approximation of the semilinear PDEs in (\ref{INTROTHM_PDE}). The central idea of the proof is, roughly speaking, that these MLP approximations can be well represented by DNNs, if the coefficients determining the linear part, the initial condition, and the nonlinearity are themselves represented by DNNs (cf. Lemma \ref{U=R}). 
	This explains why we need to assume that the coefficients of the PDEs in (\ref{INTROTHM_PDE}) can be approximated by means of DNNs without the curse of dimensionality.
	In Lemma \ref{newfullperturbation} below we establish strong approximation error bounds between solutions of the PDEs in (\ref{INTROTHM_PDE}) and solutions of the PDEs whose coefficients are specified by corresponding DNN approximations.
	In Lemma \ref{fullerror} below we construct an artificial probability space and employ a result from \cite{hutzenthaler2020multilevel} to estimate strong error bounds for the numerical approximation of the approximating PDEs by means of MLP approximations on that space. 
	In our main result, Theorem \ref{MainTheorem} below, we combine these results and prove existence of a DNN realization on the artifical probability space that suffices the prescribed approximation accuracy $ \varepsilon\in (0,1] $ between the solution of the PDEs in (\ref{INTROTHM_PDE}) and the constructed DNN, which eventually completes the proof of Theorem \ref{INTROTHM} above.
	\noindent The remainder of this article is organized as follows: 
	In section \ref{sect2} we will establish a full error bound for  MLP approximations in which every involved function is approximated (by means of DNNs).
	Section \ref{sect3} provides a mathematical framework for deep neural networks and establishes the connection between DNNs and MLP approximations, see Lemma \ref{U=R}.
	In section \ref{sect4} we combine the findings of section \ref{sect2} and section \ref{sect3} to prove our main result, Theorem \ref{MainTheorem}, as well as  Corollary \ref{MainCorollary}.
	\subsection{An Introduction to MLP-approximations}\label{sect1.1}
	Section \ref{sect2} below develops a result which proves stability for multilevel Picard approximations against perturbations in various components.
	To give an intuition, we will first informally introduce multilevel Picard approximations.
	Consider the semilinear PDE
	\begin{equation}\label{MLPINTROPDE}
		\frac{\partial}{\partial t}u(t,x)+\langle\mu(x),(\nabla_xu)(t,x)\rangle+\frac{1}{2}\textup{Tr}\left(\sigma(x)[\sigma(x)]^*(\textup{Hess}_xu)(t,x)\right)=-f(u(t,x))
	\end{equation}
	with terminal condition $ u(T,x)=g(x) $, where $ u\in C^{1,2}([0,T]\times\R^d,\R) $, $ \mu\in C(\R^d,\R^d) $, $ \sigma\in C(\R^d,\R^{d\times d}) $, $ g\in C(\R^d,\R) $ and $ f\in C(\R,\R) $,  $ f $ being the nonlinearity.
	Our goal is to find a numerical method to solve the PDE in (\ref{MLPINTROPDE}), which does not suffer from the curse of dimensionality.
	Classical grid based approaches are not suitable for this task and we may consider meshfree techniques such as Monte Carlo methods. We thus consider a filtered probability space  $ (\Omega,\mathcal{F},\mathbb{P},(\F_t)_{t\in[0,T]}) $ \textit{satisfying the usual conditions}\footnote{Note that we say that a filtered probability space $ (\Omega, \mathcal{F},\mathbb{P},(\F_t)_{t\in[0,T]}) $ \textit{satisfies the usual conditions} if and only if it holds for all $ t\in[0,T) $ that $ \{A\in \mathcal{F}:\mathbb{P}(A)=0\}\subseteq \F_t= (\cup_{s\in(t,T]}\F_s) $.} and a standard $ (\F_t)_{t\in[0,T]} $-Brownian motion $ W\colon [0,T]\times\Omega\rightarrow\R^d $.
	As we will see in Theorem \ref{MainTheorem}, under suitable assumptions there exist for all $ t\in[0,T],x\in\R^d $ up to indistinguishability unique, $ (\F_s)_{s\in[t,T]} $-adapted  stochastic processes $ X_t^x\colon[t,T]\times \Omega\rightarrow\R^d ,$ satisfying for all $ s\in[t,T] $   that 
	\begin{equation}\label{MLPINTROSDE}
		X_{t,s}^{x}=x+\int_t^s\mu(X_{t,r}^x)dr+\int_t^s\sigma(X_{t,r}^x)dW_r
	\end{equation}
	and 
	\begin{equation}\label{MLPINTROPDE_Feynman-Kac}
		u(t,x)=\E\left[g(X_{t,T}^x)\right]+\int_t^T\E\left[f(u(r,X_{t,r}^x))\right]dr.
	\end{equation}
	The reverse direction does not work in general, because differentiability of the function in (\ref{MLPINTROPDE_Feynman-Kac}) is not necessarily given. This would need sufficient smoothness of $ \mu,\sigma,f, $ and $ g. $ Fortunately, one can prove under more relaxed conditions, that (\ref{MLPINTROPDE_Feynman-Kac}) yields a unique, so called \textit{viscosity solution} of the PDE in (\ref{MLPINTROPDE}), which is a generalized concept of classical solutions. For an introduction and a review on the theory of viscosity solutions we refer to \cite{users_guide_viscosity_solutions}.
	In particular, every classical solution of (\ref{MLPINTROPDE}) is also a viscosity solution of (\ref{MLPINTROPDE}), hence we get a one-to-one correspondence between both representations. 
	Having established this connection, we may want to employ a numerical approximation of (\ref{MLPINTROPDE_Feynman-Kac}) by means of Monte Carlo methods. 
	The second term on the right-hand side of (\ref{MLPINTROPDE_Feynman-Kac}), however, is problematic since we need to know the very thing that we wish to approximate, $ u $.
	A similar problem arises in the field of ordinary differential equations: let $ y\colon [0,T]\rightarrow \R $ satisfy $ y(0)=y_0 $ and for all $ \frakt\in (0,T) $ that $ y' =\frakf(y), $ where $ \frakf \in C(\R,\R). $
	This can be reformulated to 
	\begin{equation}\label{MLPINTRO_ODE_INT}
		y(\frakt)=y_0+\int_{0}^\frakt \frakf(y(s)) ds.
	\end{equation}
	If we assume the function $ \frakf $ to satisfy a (local) Lipschitz condition, 
	it is ensured that for sufficiently small $ \frakt\in(0,T] $, the \textit{Picard iterations}
	\begin{equation}
		\begin{aligned}
			y^{[0]}(\frakt)&=y_0,\\
			y^{[n+1]}(\frakt)&=y_0+\int_{0}^\frakt \frakf(y^{[n]}(s))ds, \quad n\in\N_0
		\end{aligned}
	\end{equation}
	converge to a unique solution of (\ref{MLPINTRO_ODE_INT}), cf., e.g., [\citen{WolfgangWalterODE}, §6]. 
	Following this idea, we reformulate (\ref{MLPINTROPDE_Feynman-Kac}) as stochastic fixed point equation:
	\begin{equation}\label{INTROMLP_SFPE}
		\begin{aligned}
			\Phi^{0}(t,x)&=\E\left[g(X_{t,T}^x)\right]\\
			\Phi^{n+1}(t,x)&=\E\left[g(X_{t,T}^x)\right]+\int_t^T\E\left[f(\Phi^{n}(s,X_{t,s}^x))\right]ds, \quad n\in\N_0.
		\end{aligned}
	\end{equation}
	Indeed, the proof of [\citen{hutzenthaler2020multilevel}, Proposition 2.2] provides  a Banach space $ (\calV, \nrm{\cdot}_{\calV}), \calV \subset \R^{[0,T]\times \R^d}, $  in which the mapping $ \Phi\colon\calV\rightarrow\calV $ satisfying for all $ t\in[0,T], x\in\R^d, v\in\calV $ that 
	\begin{equation}
		(\Phi(v))(t,x)=\E\left[g(X_{t,T}^x)\right]+\int_t^T \E\left[f(v(s,X_{t,s}^x))\right]ds
	\end{equation}
	is, under suitable assumptions, a contraction. Banach's fixed point theorem thus provides existence of a unique fixed point $ u \in \calV $ with $ \Phi(u)=u $ and proves convergence of (\ref{INTROMLP_SFPE}).
	In order to derive an approximation scheme for the right-hand side of (\ref{INTROMLP_SFPE}), we employ a telescoping sum argument and define $ \Phi^{-1}\equiv 0 $, which gives us for all $ n\in\N, t\in[0,T],x\in\R^d $
	\begin{equation}\label{INTROMLP_SFPE2}
		\begin{aligned}
			\Phi^{n+1}(t,x)&=\E\left[g(X_{t,T}^x)\right]+\int_t^T\E\left[f(\Phi^n(s,X_{t,s}^x))\right]ds\\
			&=\E\left[g(X_{t,T}^x)\right]+\int_t^T\E\left[f(\Phi^0(s,X_{t,s}^x))\right]+\sum_{l=1}^n\E\left[f(\Phi^l(s,X_{t,s}^x))-f(\Phi^{l-1}(s,X_{t,s}^x))\right]ds\\
			&=\E\left[g(X_{t,T}^x)\right]+\sum_{l=0}^n\int_t^T\E\left[f(\Phi^l(s,X_{t,s}^x))-\mathbbm{1}_\N(l)f(\Phi^{l-1}(s,X_{t,s}^x))\right]ds.\\
		\end{aligned}
	\end{equation}
	With a standard Monte Carlo approach we can approximate $ \E[g(X_{t,T}^x)] $ efficiently. For the sum on the right-hand side of above equation, note that 
	\begin{enumerate}[a)]
		\item\label{2Intro_a} one can, for all $t\in[0,T], h\in L^1([t,T],\mathcal{B}([t,T]), \lambda|_{[t,T]};\R) $, approximate $ \int_t^Th(s)d\lambda(s) $ via Monte Carlo averages by considering the probability space $ ([t,T], \mathcal{B}([t,T]), \frac{1}{T-t}\lambda|_{[t,T]}) $ and sampling independent, continuous uniformly distributed $ \tau_0,...,\tau_M\colon\Omega\rightarrow[t,T], M\in\N $. This yields \begin{equation}
			\int_t^Th(s)d\lambda(s)=(T-t)\int_t^Th(s)d\frac{1}{T-t}\lambda|_{[t,T]}(s)\approx\frac{T-t}{M}\sum_{k=0}^Mh(\tau_k).
		\end{equation}
		\item\label{2Intro_b} due to the recursive structure of (\ref{INTROMLP_SFPE2}), summands of large $ l $ are costly to compute, but thanks to convergence, tend to be of small variance. In order to achieve a prescribed approximation accuracy, these summands will only need relatively few samples for a suitable Monte-Carlo average. 
		On the other hand, summands of small $ l $ are cheap to compute, but have large variance, implying the need of relatively many samples for a Monte-Carlo average that satisfies a prescribed approximation accuracy. This is the key insight, which motivates the multilevel approach.
	\end{enumerate}
	To present the multilevel approach, we need a larger index set. Let $ \Theta =\cup_{n \in \N}\mathbb{Z}^n $. Let  $ Y_{t}^{\theta,x}\colon[t,T]\times\Omega \rightarrow\R^d ,\theta\in\Theta, t\in[0,T], x\in\R^d$ be independent Euler-Maruyama approximations of $ X_t^x $ (see (\ref{2SettingEulerMaruyama}) below for a concise definition).
	Let $ \frakT^\theta_t\colon\Omega\rightarrow[t,T] ,\theta\in\Theta, t\in[0,T],$  be independent, continuous, uniformly distributed random variables. In particular, let $ (Y_t^{\theta,x})_{\theta\in\Theta} $ and $ (\frakT_t^\theta)_{\theta\in\Theta} $ be independent. 
	Combining this with (\ref{INTROMLP_SFPE2}), \ref{2Intro_a}) and \ref{2Intro_b}) gives us  the desired multilevel Picard approximations (see (\ref{2SettingMLP}) below for a concise definition):
	\begin{equation}
		\begin{aligned}
			&U_{n,M}^\theta(t,x)= \frac{\mathbbm{1}_{\N}(n)}{M^n}\sum_{i=0}^{M^n}g(Y_{t,T}^{(\theta,0,-i),x})\\
			&\quad +\sum_{l=0}^{n-1}\frac{T-t}{M^{n-l}}\sum_{i=0}^{M^{n-l}}\left(f\circ U_{l,M}^{(\theta,l,i)}-\mathbbm{1}_{\N}(l)f\circ U_{l-1,M}^{(\theta,-l,i)}\right)\left(\frakT_t^{(\theta,l,i)}, Y_{t,\frakT_t^{(\theta,l,i)}}^{(\theta,l,i),x}\right).
		\end{aligned}
	\end{equation}
	In section \ref{sect3} we will show that deep neural networks are capable to represent MLP approximations, if $ \mu, \sigma, f, $ and $ g $ are DNN functions. For this reason, we assume in section \ref{sect2} that for a given PDE such as (\ref{MLPINTROPDE}) with coefficient functions $ \mu_1,\sigma_1, f_1, $ and $ g_1 $, we can approximate its coefficients by (DNN functions) $ \mu_2,\sigma_2, f_2, $ and $ g_2, $ respectively. 
	Thus we have to take two sources of error into account: 
	\begin{enumerate}[(i)]
		\item the perturbation error of the PDE. We will utilize stochastic tools to get an estimate of $ |u_1(t,x)-u_2(t,x)|, $ where for each $ i\in\{1,2\} $ we have 
		\begin{equation}
			u_i(t,x)=\E\left[g_i(X_{t,T}^{x,i})\right]+\int_t^T\E\left[f(u_i(s,X_{t,s}^{x,i}))\right]ds,
		\end{equation}
		$ X_{t}^{x,i} $ solving the stochastic differential equation (SDE) with drift $ \mu_i $ and diffusion $ \sigma_i $ with respect to W, starting in $ (t,x)\in[0,T]\times\R^d $ (cf. (\ref{MLPINTROSDE}), (\ref{MLPINTROPDE_Feynman-Kac}), and Lemma \ref{newfullperturbation} below).
		\item the MLP error $ \left(\E\left[\left|\mathscr{U}_{n,M}^0(t,x)-u_2(t,x)\right|^2\right]\right)^{\frac{1}{2}}, $ where $ \mathscr{U}_{n,M}^\theta, \theta\in\Theta, $ are the MLP approximations built by the coefficient functions $ \mu_2,\sigma_2,f_2, $ and $ g_2$.
	\end{enumerate}
	With that, the triangle inequality enables us to find a boundary for the full error:
	\begin{equation}
		\left(	\E\left[\left|\mathscr{U}_{n,M}^0(t,x)-u_1(t,x)\right|^2\right]\right)^{\frac{1}{2}}\le\left(\E\left[\left|\mathscr{U}_{n,M}^0(t,x)-u_2(t,x)\right|^2\right]\right)^{\frac{1}{2}}+\left|u_2(t,x)-u_1(t,x)\right|.
	\end{equation}

	\section{A perturbation result for MLP-approximations}\label{sect2}
	Throughout this section, let $ \langle\cdot,\cdot\rangle\colon \cup_{d \in \N}\R^d\times\R^d\rightarrow \R, $  $\nrm{\cdot}\colon\cup_{d \in \N}\R^d\rightarrow[0,\infty),  $ $ \nrm{\cdot}_F\colon \cup_{d,m \in \N}\R^{d\times m}\rightarrow[0,\infty) $ satisfy for all $ d,m\in\N, $ $ x=(x_1,...,x_d), y=(y_1,...,y_d) \in\R^d, $ $ A= (a_{i,j})_{i\in \{1,...,d\},j\in \{1,...,m\}} \in \R^{d\times m} $ that $ \langle x,y\rangle = \sum_{i=1}^d x_iy_i$, $ \nrm{x}=\left(\sum_{i=1}^d |x_i|^2\right)^{\frac{1}{2}} $, $ \nrm{A}_F=\left(\sum_{i=1}^d\sum_{j=1}^m |a_{i,j}|^2\right)^{\frac{1}{2}}. $ 
	\subsection{An estimate for Lyapunov-type functions.}
	We start  with an auxiliary result, which states that we can raise certain Lyapunov-type functions to any power, without losing its growth property.
\begin{lemma}\label{LyapunovLemmaNew}
	\begin{enumerate}[(i)]
		\item 		Let $ d,m\in \N, c,\kappa\in [1,\infty),  p\in [2,\infty),$ let $ \varphi \in C^2(\R^d,[1,\infty)), \mu\in C(\R^d,\R^d), \sigma \in C(\R^d, \R^{d\times m}) $ satisfy for all $ x,y\in \R^d, z\in\R^d\setminus\{0\}  $ that
		\begin{equation}\label{LyapunovLemma1}
			\max\left \{ \frac{\langle (\nabla \varphi)(x),z\rangle }{(\varphi(x))^{\frac{p-1}{p}}\nrm{z}}, \frac{\langle z,(\textup{Hess}\varphi(x))z\rangle}{(\varphi(x))^{\frac{p-2}{p}}\nrm{z}^2}, \frac{c\nrm{x}+\max\{\nrm{\mu(0)},\nrm{\sigma(0)}_F\}}{(\varphi(x))^{\frac{1}{p}}} \right \}\le c,
		\end{equation}
		and \begin{equation}\label{LyapunovLemma2}
			\max\{\nrm{\mu(x)-\mu(y)}, \nrm{\sigma(x)-\sigma(y)}_F\}\le c\nrm{x-y}.
		\end{equation}
		Then it holds for all $ x\in \R^d $ that
		\begin{equation}
			\left|\langle \nabla (\varphi(x))^\kappa, \mu(x)\rangle\right|+\frac{1}{2}\textup{Tr}(\sigma(x)[\sigma(x)]^*(\textup{Hess}(\varphi(x))^\kappa)) \le \frac{\kappa}{2}((\kappa-1)c^4+3c^3)(\varphi(x))^\kappa.
		\end{equation}
		\item Additionally let $ T\in(0,\infty) $, $ (\Omega,\mathcal{F},\mathbb{P},(\F_t)_{t\in[0,T]}) $ be a filtered probability space satisfying the usual conditions, let $ W\colon[0,T]\times \Omega\rightarrow\R^m $ be a standard $ (\F_t)_{t\in[0,T]} $-Brownian motion, for all $ x\in \R^d,t\in[0,T] $ let $ X_{t}^x\colon [t,T]\times \Omega\rightarrow\R^d $ be an $ (\F_s)_{s\in[t,T]} $-adapted stochastic process satisfying $ \mathbb{P} $-a.s. for all $ s\in [t,T] $
		\begin{equation}\label{LyapunovLemma3}
			X_{t,s}^x=x+\int_t^s\mu(X_{t,r}^x)dr+\int_t^s\sigma(X_{t,r}^x)dW_r.
		\end{equation}
		Then it holds for all $ t\in [0,T], s\in [t,T], x\in \R^d $ that
		\begin{equation}
			\E\left[(\varphi(X_{t,s}^x))^\kappa\right]\le e^{\frac{1}{2}\kappa((\kappa-1)c^4+3 c^3)(s-t)}(\varphi(x))^\kappa.
		\end{equation}
	\end{enumerate}
\end{lemma}
\begin{proof}
	Note that for all $ x\in\R^d $ it holds that $ \nabla(\varphi(x))^\kappa =\kappa \varphi^{\kappa-1}(x)(\nabla\varphi)(x)$ and 
	\begin{equation}
		\textup{Hess}(\varphi(x))^\kappa=\kappa(\kappa-1)(\varphi(x))^{\kappa-2} \nabla(\varphi(x))[\nabla(\varphi(x))]^*+\kappa (\varphi(x))^{\kappa-1}\textup{Hess}\varphi(x).
	\end{equation}
	The triangle inequality, (\ref{LyapunovLemma1}), and (\ref{LyapunovLemma2}) ensure for all $ x\in \R^d $ that
	\begin{equation}
		\begin{aligned}
			\max\{\nrm{\mu(x)},\nrm{\sigma(x)}_F\}&\le c\nrm{x}+\max\{\nrm{\mu(0)}, \nrm{\sigma(0)}_F\} \le c (\varphi(x))^{\frac{1}{p}}.
		\end{aligned}
	\end{equation}
	This and  (\ref{LyapunovLemma1}) ensure for all $ x\in\R^d $  that
	\begin{equation}
		\begin{aligned}
			\textup{Tr}(\sigma(x)[\sigma(x)]^*\textup{Hess}\varphi(x))&\le c \nrm{\sigma(x)}_F^2(\varphi(x))^{\frac{p-2}{p}} \le c^3 \varphi(x).
		\end{aligned}
	\end{equation}
	The Cauchy-Schwarz inequality and (\ref{LyapunovLemma1}) ensure for all $ x\in \R^d  $ that 
	\begin{equation}
		\begin{aligned}
			\textup{Tr}(\sigma(x)[\sigma(x)]^*\nabla\varphi(x)[\nabla\varphi(x)]^*)
			&\le \nrm{\sigma(x)[\sigma(x)]^*}_F\nrm{\nabla\varphi(x)[\nabla\varphi(x)]^*}_F\\
			&\le \nrm{\sigma(x)}^2_F\nrm{\nabla\varphi(x)}^2 \le c^4(\varphi(x))^2.
		\end{aligned}
	\end{equation}
	It thus holds for all $ x\in \R^d $ that 
	\begin{equation}
		\begin{aligned}
			\textup{Tr}(\sigma(x)[\sigma(x)]^*\textup{Hess}((\varphi(x))^\kappa))&=\kappa(\kappa-1)(\varphi(x))^{\kappa-2}\textup{Tr}(\sigma(x)[\sigma(x)]^*\nabla\varphi(x)[\nabla\varphi(x)]^*)\\
			&\quad + \kappa(\varphi(x))^{\kappa-1}\textup{Tr}(\sigma(x)[\sigma(x)]^*\textup{Hess}\varphi(x))\\
			&\le \kappa(\kappa-1)c^4(\varphi(x))^\kappa+ \kappa c^3(\varphi(x))^{\kappa}.
		\end{aligned}
	\end{equation}
	This and (\ref{LyapunovLemma1}) ensure for all $ x\in \R^d $ that 
	\begin{equation}
		\begin{aligned}
			&\left|\langle \nabla (\varphi(x))^\kappa, \mu(x)\rangle\right|+\frac{1}{2}\textup{Tr}(\sigma(x)[\sigma(x)]^*\textup{Hess}(\varphi(x))^\kappa)\\
			&\le \kappa (\varphi(x))^\kappa\left|\langle(\nabla\varphi)(x),\mu(x)\rangle\right|+\frac{1}{2}\kappa((\kappa-1)c^4+c^3)(\varphi(x))^\kappa\\
			&\le c^2\kappa (\varphi(x))^\kappa +\frac{1}{2}\kappa((\kappa-1)c^4+c^3)(\varphi(x))^\kappa\\
			&\le (\frac{1}{2}\kappa(\kappa-1)c^4+\frac{3}{2}\kappa c^3)(\varphi(x))^\kappa,
		\end{aligned}
	\end{equation}
	which establishes item $ (i) $. This and [\citen{cox2021local}, Lemma 2.2] (applied for every $ x\in \R^d, t\in [0,T], s\in [t,T] $ with $ T\curvearrowleft T-t,$  $ O \curvearrowleft \R^d,$  $ V \curvearrowleft ([0,T-t]\times \R^d \ni (s,x)\rightarrow(\varphi(x))^\kappa \in [1,\infty)), $ 
	$ \alpha \curvearrowleft ( [0,T-t]\ni s\rightarrow(\frac{1}{2}\kappa(\kappa-1)c^4+\frac{3}{2}\kappa c^3)\in [0,\infty) ), \tau \curvearrowleft s-t, X \curvearrowleft (X_{t,t+r}^x)_{r\in [0,T-t]},$ in the notation of [\citen{cox2021local}, Lemma 2.2]) proves that for all $ x\in \R^d, t\in [0,T], s\in [t,T] $ we have that 
	\begin{equation}
		\E\left[\varphi(X_{t,s}^x)\right]\le e^{\frac{1}{2}\kappa((\kappa-1)c^4+3c^3)(s-t)}(\varphi(x))^\kappa.
	\end{equation}
	This establishes item $ (ii) $ and thus finishes the proof.
\end{proof}
\subsection{A perturbation estimate for solutions of PDEs}
\begin{setting}\label{setting2new}
	Let $ d,m \in \N, $ $ \delta, T \in (0,\infty), $ $ \beta, b, c \in [1,\infty),  $ $ q\in [2,\infty), p\in [2\beta, \infty), $  
	for  all $ i\in\{1,2\} $ let $ \varphi\in C^2(\R^d,[1,\infty)),$ $ g_i\in C(\R^d,\R), \mu_i \in C(\R^d,\R^d), \sigma_i \in C(\R^d,\R^{d\times m}), $ $ f_i \in C(\R,\R), F_i\colon \R^{[0,T]\times\R^d}\rightarrow \R^{[0,T]\times \R^d} $ satisfy for all $ t\in [0,T], v,w\in \R, x,y\in\R^d, z\in \R^d\setminus\{0\}, u\colon[0,T]\times \R^d\rightarrow \R $ that $ (F_i(u))(t,x) =f_i(t,x,u(t,x)),$
	\begin{equation}\label{setting2new_Lyapunov}
		\max\{ \frac{\langle (\nabla \varphi)(x),z\rangle }{(\varphi(x))^{\frac{p-1}{p}}\nrm{z}}, \frac{\langle z,(\textup{Hess}\varphi(x))z\rangle}{(\varphi(x))^{\frac{p-2}{p}}\nrm{z}^2}, \frac{c\nrm{x}+\max\{\nrm{\mu_i(0)},\nrm{\sigma_i(0)}_F\}}{(\varphi(x))^{\frac{1}{p}}} \}\le c,
	\end{equation}
	\begin{equation}\label{setting2new_growth_g_f}
		\max\{T|f_i(t,x,0)|, |g_i(x)|\}\le b(\varphi(x))^{\frac{\beta}{p}},
	\end{equation}
	\begin{equation}\label{setting2new_localLipschitz_g_f}
		\max\{ |g_i(x)-g_i(y)|, T|f_i(t,x,v)-f_i(t,y,w)|\}\le cT|v-w|+ \frac{(\varphi(x)+\varphi(y))^{\frac{\beta}{p}}\nrm{x-y}}{T^{\frac{1}{2}}b^{-1}},
	\end{equation}
	\begin{equation}\label{setting2new_Lipschitz_mu_sigma}
		\max\{\nrm{\mu_i(x)-\mu_i(y)}, \nrm{\sigma_i(x)-\sigma_i(y)}_F\}\le c\nrm{x-y},
	\end{equation}
	\begin{equation}\label{setting2new_g_f_mu_sigma_approx}\hspace{-.25cm}
		\max\{|f_1(t,x,v)-f_2(t,x,v)|,|g_1(x)-g_2(x)|, \nrm{\mu_1(x)-\mu_2(x)}, \nrm{\sigma_1(x)-\sigma_2(x)}_F\}\le \delta ((\varphi(x))^q+|v|^q),
	\end{equation}
	let $ (\Omega, \mathcal{F}, \mathbb{P},(\F_t)_{t\in[0,T]}) $ be a filtered probability space satisfying the usual conditions, let $ W\colon [0,T]\times \Omega \rightarrow \R^m $ be a standard $ (\F_t)_{t\in[0,T]} $-Brownian motion, for all $ x\in \R^d, t\in [0,T],i\in \{1,2\} $ let $ X_t^{x,i}\colon [t,T]\times\Omega\rightarrow \R^d $ be an $ (\F_s)_{s\in[t,T]} $-adapted stochastic process satisfying for all $ s\in [t,T] $ $ \mathbb{P} $-a.s. 
	\begin{equation}\label{setting2new_X}
		X_{t,s}^{x,i}=x+\int_t^s \mu_i(X_{t,r}^{x,i})dr+\int_t^s \sigma_i(X_{t,r}^{x,i})dW_r,
	\end{equation}
	for all $ t\in [0,T], s\in [t,T] , r\in [s,T], x\in \R^d, i\in \{1,2\}$ let 
	\begin{equation}\label{setting2new_X_pathwiseUnique}
		\mathbb{P}\left(X_{s,r}^{X_{t,s}^{x,i},i}=X_{t,r}^{x,i}\right)=1,
	\end{equation}
	let $ u_1,u_2\in C([0,T]\times \R^d,\R) ,$ assume for all $ i\in \{1,2\}, t\in[0,T], x\in\R^d $ that
	\begin{equation}\label{setting2new_u_intble}
		\left(\sup_{s\in[0,T]}\sup_{y\in\R^d}\frac{|u_i(s,y)|}{\varphi(y)} \right)+\E\left[\left|g_i(X_{t,T}^{x,i})\right| \right]  +\int_t^T\E\left[ \left|(F_i(u_i))(s,X_{t,s}^{x,i})\right| \right] ds<\infty,
	\end{equation}
	\begin{equation}\label{setting2new_u}
		u_i(t,x)=\E\left[g_i(X_{t,T}^{x,i})\right]+\int_t^T\E\left[(F_i(u_i))(s,X_{t,s}^{x,i})\right]ds.
	\end{equation}
\end{setting}
\begin{lemma}\label{u_moment_new}
	Assume Setting \ref{setting2new} and let $x\in\R^d, t\in[0,T], s\in [t,T] $. Then it holds that
	\begin{equation}
		\E\left[\left| u_2(s,X_{t,s}^{x,2}) \right|^q\right]\le 2^{q-1}b^q(T+1)^q (\varphi(x))^{\frac{\beta q}{p}}e^{\frac{\beta q}{2p}((q-1)c^4+3c^3)(T-t)+c^q(T-s)^q}.
	\end{equation}
\end{lemma}
\begin{proof}
	Display (\ref{setting2new_u_intble}) and Lemma \ref{LyapunovLemmaNew} $ (ii) $ (applied with $ \kappa\curvearrowleft q $ in the notation of Lemma \ref{LyapunovLemmaNew}) ensure  that 
	\begin{equation}\label{u_moment_proof1}
		\begin{aligned}
			\left(\E\left[|u_2(s,X_{t,s}^{x,2})|^q\right]\right)^{\frac{1}{q}}&\le \left(\E\left[\left(\frac{|u_2(s,X_{t,s}^{x,2})|}{\varphi(X_{t,s}^{x,2})}\varphi(X_{t,s}^{x,2})\right)^q\right]\right)^{\frac{1}{q}}\\
			&\le \left[\sup_{r\in[t,T]}\sup_{y\in\R^d}\frac{|u_2(r,y)|}{\varphi(y)}\right]\left(\E\left[\left(\varphi(X_{t,s}^{x,2})\right)^q\right]\right)^{\frac{1}{q}}\\
			&<\infty.
		\end{aligned}
	\end{equation}
	Note  that (\ref{setting2new_growth_g_f}), Jensen's inequality, and Lemma \ref{LyapunovLemmaNew} $ (ii) $ (applied with $ \kappa\curvearrowleft q $ in the notation of Lemma \ref{LyapunovLemmaNew}) imply that 
	\begin{equation}\label{u_moment_proof2}
		\begin{aligned}
			\left(\E\left[\left|g_2(X_{t,T}^{x,2})\right|^q\right]\right)^{\frac{1}{q}}&\le b\left(\E\left[\left(\varphi(X_{t,T}^{x,2})\right)^{\frac{\beta q}{p}}\right]\right)^{\frac{1}{q}}\\
			&\le b\left(  \E\left[\left(\varphi(X_{t,T}^{x,2})\right)^q\right] \right)^{\frac{\beta}{qp}}\\
			&\le  b e^{\frac{\beta}{2p}((q-1)c^4+3c^3)(T-t)}(\varphi(x))^{\frac{\beta}{p}}\\
		\end{aligned}
	\end{equation}
	The triangle inequality, (\ref{setting2new_localLipschitz_g_f}), (\ref{setting2new_growth_g_f}),  and Lemma \ref{LyapunovLemmaNew} $ (ii) $ (applied with $ \kappa \curvearrowleft \frac{q}{2} $ in the notation of Lemma \ref{LyapunovLemmaNew}) imply 
	\begin{equation}\label{u_moment_proof3}
		\begin{aligned}
			&\left(\int_s^T\E\left[\left|(F_2(u_2))(r,X_{t,r}^{x,2})\right|^q\right]dr\right)^{\frac{1}{q}}\\
			& \le \left(\int_s^T\E\left[\left|f_2(r,X_{t,r}^{x,2},u_2(r,X_{t,r}^{x,2}))-f_2(r,X_{t,r}^{x,2},0)\right|^q\right]dr\right)^{\frac{1}{q}}+ \left(\int_s^T\E\left[\left|f_2(r,X_{t,r}^{x,2},0)\right|^q\right]dr\right)^{\frac{1}{q}}\\
			&\le c\left( \int_s^T\E\left[\left|u_2(r,X_{t,r}^{x,2})\right|^q\right]dr  \right)^{\frac{1}{q}}+\frac{b}{T}\left(\int_s^T\E\left[(\varphi(X_{t,r}^{x,2}))^{\frac{\beta q}{p}}\right]dr\right)^{\frac{1}{q}}\\
			&\le c\left( \int_s^T\E\left[\left|u_2(r,X_{t,r}^{x,2})\right|^q\right]dr  \right)^{\frac{1}{q}}+\frac{b(T-s)^{\frac{1}{q}}}{T}\left(\sup_{r\in[s,T]}\E\left[(\varphi(X_{t,r}^{x,2}))^{\frac{\beta q}{p}}\right]\right)^{\frac{1}{q}}\\
			&\le  c\left( \int_s^T\E\left[\left|u_2(r,X_{t,r}^{x,2})\right|^q\right]dr  \right)^{\frac{1}{q}}+\frac{b(T-s)^{\frac{1}{q}}}{T}e^{\frac{\beta}{2p}((q-1)c^4+3c^3)(T-t)}(\varphi(x))^{\frac{\beta}{p}}.
		\end{aligned}
	\end{equation}
	This, (\ref{setting2new_X_pathwiseUnique}), the definition of $ u_2 $, the triangle inequality, Jensen's inequality, Fubini's theorem, and (\ref{u_moment_proof2}) ensure that 
	\begin{equation}\hspace{-0cm}
		\begin{aligned}
			&\E\left[|u_2(s,X_{t,s}^{x,2})|^q\right]=\left( \E\left[\left |g_2\left (X_{s,T}^{X_{t,s}^{x,2},2}\right )+\int_s^T(F_2(u_2))(r,X_{s,r}^{X_{t,s}^{x,2},2})dr\right |^q\right] \right)^{\frac{q}{q}}\\
			&\le \left(\left(\E\left[\left | g_2(X_{t,T}^{x,2})\right |^q\right]\right)^{\frac{1}{q}}+ (T-s)^{\frac{q-1}{q}}\left(\int_s^T\E\left[\left| (F_2(u_2))(r,X_{t,r}^{x,2})  \right |^q\right]dr\right)^{\frac{1}{q}} \right)^q\\
			&\le \left((T-s)^{\frac{q-1}{q}} c\left(\int_s^T \E\left[\left|u_2(r,X_{t,r}^{x,2})\right|^q\right]dr\right)^{\frac{1}{q}} +2b e^{\frac{\beta}{2p}((q-1)c^4+3c^3)(T-t)}(\varphi(x))^{\frac{\beta}{p}} \right)^q\\
			&\le 2^{q-1}(T-s)^{q-1}c^q\int_s^T\E\left[\left|u_2(r,X_{t,r}^{x,2})\right|^q\right]dr + 2^qb^q e^{\frac{\beta q}{2p}((q-1)c^4+3c^3)(T-t)}(\varphi(x))^{\frac{\beta q}{p}}.
		\end{aligned}
	\end{equation}
	This, (\ref{u_moment_proof1}), and Gronwall's inequality  yield
	\begin{equation}
		\begin{aligned}
			\E\left[\left|u_2(s,X_{t,s}^{x,2})\right|^q\right]&\le 2^{q}b^qe^{\frac{\beta q}{2p}((q-1)c^4+3c^3)(T-t)+2^{q-1}c^q(T-s)^q}\left(\varphi(x)\right)^{\frac{\beta q}{p}}.
		\end{aligned}
	\end{equation}
	This finishes the proof.
\end{proof}
\begin{lemma}\label{newfullperturbation}
	Assume Setting \ref{setting2new}. Then it holds for all $t\in [0,T], x\in \R^d $
	\begin{equation}
		\sup_{s\in[t,T]}\E\left[\left|u_2(s,X_{t,s}^{x,2})-u_1(s,X_{t,s}^{x,1})\right|\right]\le \delta 2^{q+2}b^q(T+1)e^{q(2qc^4+3c^3)(T-t)+2^qc^qT^q+(c+1)^2}(\varphi(x))^{q+\frac{1}{2}}.
	\end{equation}
\end{lemma}
\begin{proof}
	First, [\citen{Hutzenthaler_2020_perturbation}, Theorem  1.2] (applied for every $ t\in[0,T], s\in [t,T], x\in\R^d $ with $ H\curvearrowleft \R^d, U\curvearrowleft\R^d, D\curvearrowleft\R^d, T\curvearrowleft s-t, (\F_r)_{r\in[0,T]}\curvearrowleft (\F_{r+t})_{r\in[0,s-t]}, (W_r)_{r\in[0,T]}\curvearrowleft(W_{t+r}-W_t)_{r\in[0,s-t]}, (X_r)_{r\in[0,T]}\curvearrowleft (X_{t,t+r}^{x,1})_{r\in[0,s-t]}, (Y_r)_{r\in[0,T]}\curvearrowleft (X_{t,t+r}^{x,2})_{r\in[0,s-t]}, \mu\curvearrowleft\mu_1, \sigma\curvearrowleft\sigma_1,$  $ (a_r)_{r\in[0,T]}\curvearrowleft (\mu_2(X_{t,t+r}^{x,2}))_{r\in[0,s-t]},$  $ (b_r)_{r\in[0,T]}\curvearrowleft (\sigma_2(X_{t,t+r}^{x,2}))_{r\in[0,s-t]}, $  $ \tau\curvearrowleft(\Omega\ni\omega\rightarrow s-t\in[0,s-t]),  p\curvearrowleft2,q\curvearrowleft\infty, r\curvearrowleft 2, \alpha \curvearrowleft1, \beta\curvearrowleft1, \varepsilon\curvearrowleft1 $ in the notation of [\citen{Hutzenthaler_2020_perturbation}, Theorem 1.2]), combined with the Cauchy-Schwartz inequality, (\ref{setting2new_Lipschitz_mu_sigma}), and (\ref{setting2new_g_f_mu_sigma_approx}) show that for all $ t\in[0,T], s\in[t,T], x\in\R^d $ we have 
	\begin{equation}
		\begin{aligned}
			&\left(\E\left[\nrm{X_{t,s}^{x,1}-X_{t,s}^{x,2}}^2\right]\right)^{\frac{1}{2}}\\
			&\le \underset{z \ne y}{\sup_{z,y\in\R^d,}}\exp\left(\left[\frac{\langle z-y,\mu_1(z)-\mu_1(y)\rangle+\nrm{\sigma_1(z)-\sigma_1(y)}_F^2}{\nrm{z-y}^2}+\frac{1}{2}\right]^+\right)\\
			&\cdot\left(\left(\int_t^s\E\left[\nrm{\mu_2(X_{t,r}^{x,2})-\mu_1(X_{t,r}^{x,2})}^2\right]dr\right)^{\frac{1}{2}}+\sqrt{2}\left(\int_t^s\E\left[\nrm{\sigma_2(X_{t,r}^{x,2})-\sigma_1(X_{t,r}^{x,2})}^2\right]dr\right)^{\frac{1}{2}}\right)\\
			&\le e^{c+c^2+\frac{1}{2}}\delta(1+\sqrt{2})\left(\int_t^s \E\left[\varphi(X_{t,r}^{x,2})^{2q}\right]dr\right)^{\frac{1}{2}}\\
			&\le e^{c+c^2+\frac{1}{2}}\delta (1+\sqrt{2}) \sqrt{s-t}\left(\sup_{r\in[t,s]}\E\left[\varphi(X_{t,r}^{x,2})^{2q}\right]\right)^\frac{1}{2}.
		\end{aligned}
	\end{equation}
	This combined with Lemma \ref{LyapunovLemmaNew} $ (ii) $ (applied with $ \kappa \curvearrowleft 2q $ in the notation of \ref{LyapunovLemmaNew}) hence demonstrates for all $ x\in \R^d, t\in [0,T],s\in[t,T] $ that
	\begin{equation}
		\begin{aligned}
			\left(\E\left[\nrm{X_{t,s}^{x,1}-X_{t,s}^{x,2}}^2\right]\right)^{\frac{1}{2}}&\le \delta 2\sqrt{2} e^{c^2+c+\frac{1}{2}}\sqrt{s-t} e^{q((2q-1)c^4+3c^3)(s-t)}\varphi(x)^q.
		\end{aligned}
	\end{equation}
	This and Hölder's inequality ensure for all $ x\in \R^d, t\in [0,T],s\in[t,T] $ that
	\begin{equation}\label{fullperturbation_proof_1}
		\begin{aligned}
			&\E\left[(\varphi(X_{t,s}^{x,1})+\varphi(X_{t,s}^{x,2}))^{\frac{\beta}{p}}\nrm{X_{t,s}^{x,1}-X_{t,s}^{x,2}}\right]\\
			&\le \E\left[(\varphi(X_{t,s}^{x,1})+\varphi(X_{t,s}^{x,2}))^{\frac{1}{2}}\nrm{X_{t,s}^{x,1}-X_{t,s}^{x,2}}\right]\\
			&\le \left(\E\left[\varphi(X_{t,s}^{x,1})+\varphi(X_{t,s}^{x,2})\right]\right)^{\frac{1}{2}}\left(\E\left[\nrm{X_{t,s}^{x,1}-X_{t,s}^{x,2}}^2\right]\right)^{\frac{1}{2}}\\
			&\le \sqrt{2}e^{\frac{3}{4}c^3(s-t)}(\varphi(x))^{\frac{1}{2}}\left(\E\left[\nrm{X_{t,s}^{x,1}-X_{t,s}^{x,2}}^2\right]\right)^{\frac{1}{2}}\\
			&\le \delta 4 \sqrt{s-t}e^{(q((2q-1)c^4+3c^3))(s-t)+c^2+c+\frac{1}{2}}(\varphi(x))^{q+\frac{1}{2}}.
		\end{aligned}
	\end{equation}
	Next,  the triangle inequality and (\ref{setting2new_X_pathwiseUnique}) guarantee for all $ t\in[0,T], s\in[t,T], x\in\R^d $ that
	\begin{equation}
		\begin{aligned}
			&\E\left[\left |u_2(s,X_{t,s}^{x,2})-u_1(s,X_{t,s}^{x,1})\right |\right]\\
			&\le \E\left[\left|g_2\left (X_{s,T}^{X_{t,s}^{x,2},2}\right )-g_1\left (X_{s,T}^{X_{t,s}^{x,1},1}\right )\right|\right]+\int_s^T\E\left[\left|\left(F_2(u_2)\right)\left(r,X_{s,r}^{X_{t,s}^{x,2},2}\right)-\left(F_1(u_1)\right)\left(r,X_{s,r}^{X_{t,s}^{x,1},1}\right)\right|\right]dr\\
			&=\E\left[\left|g_2(X_{t,T}^{x,2})-g_1(X_{t,T}^{x,1})\right|\right]+\int_s^T\E\left[\left|(F_2(u_2))\left(r,X_{t,r}^{x,2}\right)-(F_1(u_1))\left(r,X_{t,r}^{x,1}\right)\right|\right]dr\\
			&\le \E\left[\left|g_2(X_{t,T}^{x,2})-g_1(X_{t,T}^{x,2})\right|\right]+\int_s^T\E\left[\left|(F_2(u_2))\left(r,X_{t,r}^{x,2}\right)-(F_1(u_2))\left(r,X_{t,r}^{x,2}\right)\right|\right]dr    \\
			&\quad+\E\left[\left|g_1(X_{t,T}^{x,2})-g_1(X_{t,T}^{x,1})\right|\right]+\int_s^T\E\left[\left|(F_1(u_2))\left(r,X_{t,r}^{x,2}\right)-(F_1(u_1))\left(r,X_{t,r}^{x,1}\right)\right|\right]dr.
		\end{aligned}
	\end{equation}
	This, (\ref{setting2new_localLipschitz_g_f}), and (\ref{setting2new_g_f_mu_sigma_approx}) ensure for all $ t\in[0,T],s\in[t,T], x\in \R^d $ that
	\begin{equation}
		\begin{aligned}
			&\E\left[\left |u_2(s,X_{t,s}^{x,2})-u_1(s,X_{t,s}^{x,1})\right |\right]\\
			&\le T^{-\frac{1}{2}}b\E\left[\left(\varphi(X_{t,T}^{x,1})+\varphi(X_{t,T}^{x,2})\right)^{\frac{\beta}{p}}\nrm{X_{t,T}^{x,1}-X_{t,T}^{x,2}}\right]+\delta \E\left[(\varphi(X_{t,T}^{x,2}))^q\right]\\
			&+ \delta\int_s^T\E\left[\varphi(X_{t,r}^{x,2})^q\right]+\left|u_2(r,X_{t,r}^{x,2})\right|^qdr +c\int_s^T\E\left[\left|u_2(r,X_{t,r}^{x,2})-u_1(r,X_{t,r}^{x,1})\right|\right]dr\\ &+bT^{-\frac{1}{2}}\left(\sup_{r\in[s,T]}\E\left[\left(\varphi(X_{t,r}^{x,1})+\varphi(X_{t,r}^{x,2})\right)^{\frac{\beta}{p}}\nrm{X_{t,r}^{x,1}-X_{t,r}^{x,2}}\right]\right)\\
			&\le 2bT^{-\frac{1}{2}}\left(\sup_{r\in[s,T]}\E\left[\left(\varphi(X_{t,T}^{x,1})+\varphi(X_{t,T}^{x,2})\right)^{\frac{\beta}{p}}\nrm{X_{t,T}^{x,1}-X_{t,T}^{x,2}}\right]\right)\\
			&+\delta(T-s+1)\left( \sup_{r\in[s,T]}\E\left[(\varphi(X_{t,r}^{x,2}))^q\right]\right)+\delta(T-s) \left(\sup_{r\in[s,T]}\E\left[\left|u_2(r,X_{t,r}^{x,2})\right|^q\right]\right)\\
			&+c\int_s^T\E\left[\left|u_2(r,X_{t,r}^{x,2})-u_1(r,X_{t,r}^{x,1})\right|\right]dr.
		\end{aligned}
	\end{equation}
	This, Lemma \ref{LyapunovLemmaNew} $ (ii) $ (applied with $ \kappa\curvearrowleft q $ in the notation of Lemma \ref{LyapunovLemmaNew}), (\ref{fullperturbation_proof_1}), and Lemma \ref{u_moment_new} hence ensure for all $ x\in \R^d, t\in [0,T], s\in[t,T] $ that
	\begin{equation}
		\begin{aligned}
			\E\left[ \left| u_2(s,X_{t,s}^{x,2})-u_1(s,X_{t,s}^{x,1}) \right| \right]&\le \delta 8 e^{(q((2q-1)c^4+3c^3))(T-t)+c^2+c+\frac{1}{2}}(\varphi(x))^{q+\frac{1}{2}}\\
			&\quad +\delta(T-s+1) e^{\frac{q}{2}((q-1)c^4+3c^3)(T-t)}(\varphi(x))^q\\
			&\quad + \delta(T-s) 2^qb^qe^{\frac{q}{4}((q-1)c^4+3c^3)(T-t)+2^{q-1}c^q(T-t)^q}(\varphi(x))^{\frac{q}{2}}\\
			&\quad+ c\int_s^T\E\left[\left|u_2(r,X_{t,r}^{x,2})-u_1(r,X_{t,r}^{x,1})\right|\right]dr\\
			&\le\delta 2^{q+2}b^q(T+1)e^{(q((2q-1)c^4+3c^3))(T-t)+2^qc^qT^q+(c+1)^2}(\varphi(x))^{q+\frac{1}{2}}\\
			&\quad +c\int_s^T\E\left[\left|u_2(r,X_{t,r}^{x,2})-u_1(r,X_{t,r}^{x,1})\right|\right]dr\\
		\end{aligned}
	\end{equation}
	This and Gronwall's inequality ensure for all $ x\in\R^d,t\in[0,T], s\in[t,T] $ that
	\begin{equation}
		\E\left[ \left| u_2(s,X_{t,s}^{x,2})-u_1(s,X_{t,s}^{x,1}) \right| \right]\le \delta 2^{q+2}b^q(T+1)e^{q(2qc^4+3c^3)(T-t)+2^qc^qT^q+(c+1)^2}(\varphi(x))^{q+\frac{1}{2}}.
	\end{equation}
	The proof is thus finished.
\end{proof}
\subsection{A full error estimate}
The following corollary combines Lemma \ref{newfullperturbation} with [\citen{hutzenthaler2020multilevel}, Theorem 4.2] and yields an error estimate for MLP-approximations, where the PDE coefficients are perturbed.
\begin{corollary}\label{fullerror}
		Assume Setting \ref{setting2new}, let $ \Theta=\cup_{n\in\N}\mathbb{Z}^n$, let $ \frakt^\theta\colon\Omega\rightarrow[0,1] ,\theta\in\Theta,$ be i.i.d. random variables satisfying for all $ t\in (0,1) $  that $ \mathbb{P}(\frakt^0\le t)=t $, let $ \frakT^\theta\colon[0,T]\times\Omega\rightarrow[0,T] , \theta\in\Theta,$ satisfy for all $ t\in[0,T] $ that $ \frakT^\theta_t=t+(T-t)\frakt^\theta, $ let $ W^\theta\colon[0,T]\times\Omega\rightarrow\R^d,\theta\in\Theta,$ be independent standard $  (\F_t)_{t\in[0,T]} $-Brownian motions,  assume indepence of $ (W^\theta)_{\theta\in\Theta}$ and $ (\frakt^\theta)_{\theta \in \Theta} $, 
	for every $ N\in\N, \theta\in\Theta, x\in\R^d, t\in[0,T] $ let $ Y^{N,\theta,x}_t\colon[t,T]\times \Omega\rightarrow\R^d $ satisfy for all $ n\in \{0,1,...,N\} $, $ s\in \left[\frac{nT}{N},\frac{(n+1)T}{N}\right]\cap[t,T] $ that $ Y_{t,t}^{N,\theta,x}=x $ and 
	\begin{equation}\label{2SettingEulerMaruyama}
		Y_{t,s}^{N,\theta,x}-Y_{t,\max\{t,\frac{nT}{N}\}}^{N,\theta,x}=\mu_2(Y_{t,\max\{t,\frac{nT}{N}\}}^{N,\theta,x})(s-\max\{t,\frac{nT}{N}\})+\sigma_2(Y_{t,\max\{t,\frac{nT}{N}\}}^{N,\theta,x})(W_s^\theta-W_{\max\{t,\frac{nT}{N}\}}^\theta),
	\end{equation}
	let $ U_{n,M}^\theta\colon[0,T]\times \R^d\times \Omega\rightarrow\R, n,M\in\mathbb{Z},\theta\in\Theta, $ satisfy for all $ \theta\in\Theta, n\in\N_0, M\in\N, t\in[0,T], x\in\R^d $ that 
	\begin{equation}\label{2SettingMLP}
		\begin{aligned}
			&	U_{n,M}^\theta(t,x)=\frac{\mathbbm{1}_\N(n)}{M^n}\sum_{i=1}^{M^n}g_2\left(Y_{t,T}^{M^M,(\theta,0,-i),x}\right)\\
			&+\sum_{l=0}^{n-1}\frac{T-t}{M^{n-l}}\left[\sum_{i=1}^{M^{n-l}}\left(F_2\left(U_{l,M}^{(\theta,l,i)}\right)-\mathbbm{1}_\N(l)F_2\left(U_{l-1,M}^{(\theta,-l,i)}\right)\right)\left(\frakT_t^{(\theta,l,i)},Y_{t,\frakT_t^{(\theta,l,i)}}^{M^M,(\theta,l,i),x}\right)\right].
		\end{aligned}
	\end{equation}
	Then it holds for all $ n\in\N_0, M\in\N, t\in[0,T],x\in\R^d $ that
	\begin{equation}
		\begin{aligned}
			\left (\E\left[\left|U_{n,M}^{0}(t,x)-u_1(t,x)\right|^2\right]\right )^{\frac{1}{2}}&\le 4^qb^qc^2(T+1)e^{q(qc^4+3c^3)T+2^qc^qT^q+(c+1)^2}(\varphi(x))^q\\
			&\quad \cdot\left[ \delta + \frac{\exp(2ncT+\frac{M}{2})}{M^{\frac{n}{2}}}+\frac{1}{M^{\frac{M}{2}}}\right].
		\end{aligned}
	\end{equation}
\end{corollary}
\begin{proof}
	First, [\citen{hutzenthaler2020multilevel}, Theorem 4.2(iv)] (applied with $ f\curvearrowleft f_2, g\curvearrowleft g_2, \mu \curvearrowleft \mu_2, \sigma \curvearrowleft \sigma_2 $ in the notation of [\citen{hutzenthaler2020multilevel}, Theorem 4.2(iv)]) ensures for all $ t\in [0,T], x\in \R^d, n\in \N_0, M\in\N $ that 
	\begin{equation}
		\left( \E\left[\left|U_{n,M}^0(t,x)-u_2(t,x)\right|^2\right] \right)^{\frac{1}{2}}\le \left[\frac{\exp(2ncT+\frac{M}{2})}{M^{\frac{n}{2}}}+\frac{1}{M^{\frac{M}{2}}}\right]12bc^2|\varphi(x)|^{\frac{\beta+1}{p}}e^{9c^3T}.
	\end{equation}
	This, the triangle inequality, and Lemma \ref{newfullperturbation} ensure for all $ t\in[0,T],x\in\R^d,n\in\N_0,M\in\N $ that
	\begin{equation}
		\begin{aligned}
			\left(\E\left[\left|U_{n,M}^0(t,x)-u_1(t,x)\right|^2\right]\right)^{\frac{1}{2}}&\le \left( \E\left[\left|U_{n,M}^0(t,x)-u_2(t,x)\right|^2\right] \right)^{\frac{1}{2}}+|u_2(t,x)-u_1(t,x)|\\
			&\le \left[\frac{\exp(2ncT+\frac{M}{2})}{M^{\frac{n}{2}}}+\frac{1}{M^{\frac{M}{2}}}\right]12bc^2|\varphi(x)|^{\frac{\beta+1}{p}}e^{9c^3T}\\
			&+ \delta 2^{q+2}b^q(T+1)e^{q(2qc^4+3c^3)(T-t)+2^qc^qT^q+(c+1)^2}(\varphi(x))^{q+\frac{1}{2}}\\
			&\le 4^qb^qc^2(T+1)e^{q(2qc^4+3c^3)T+2^qc^qT^q+(c+1)^2}(\varphi(x))^{q+\frac{1}{2}}\\
			&\quad \cdot\left[ \delta + \frac{\exp(2ncT+\frac{M}{2})}{M^{\frac{n}{2}}}+\frac{1}{M^{\frac{M}{2}}}\right].
		\end{aligned}
	\end{equation}
	This finishes the proof.
\end{proof}
	\section{Deep Neural Networks}\label{sect3}
	In this section we will prove that multilevel Picard approximations can be represented by deep neural networks (DNNs), allowing us to construct a DNN we can apply the results established in Section \ref{sect2} to. This generalizes [\citen{Hutzenthaler_2020}, Lemma 3.10] in the sense that we additionally establish a DNN representation of Euler-Maruyama approximations.
	We will first briefly introduce artificial neural networks from a mathematical point of view.
	There are many types of artificial neural network architectures, in this article however, we only consider \textit{feed-forward} neural network architectures. 
	Mathematically, a feed-forward neural network of depth $ \mathcal{L}+1 \in\N$ represents a composition of functions 
	\begin{equation}\label{INTRODNN_COMP}
		T_\mathcal{L} \circ \mathbf{A}\circ T_{\mathcal{L}-1}\circ \mathbf{A} \circ T_{\mathcal{L}-2} \circ \mathbf{A} \circ \dots \circ T_2\circ \mathbf{A}\circ T_1,
	\end{equation}
	where $ T_1,..., T_\mathcal{L} $ are affine-linear mappings and $ \mathbf{A} $ refers to a nonlinear mapping that is applied componentwise, which is called \textit{activation function}. 
	Note that the last affine-linear mapping $ T_{\mathcal{L}} $ is not followed by an activation function. 
	For every $ k\in\{1,...,\mathcal{L}-1\} $ the composition $ (\mathbf{A}\circ T_k)$ is called \textit{hidden layer}, while $ T_{\mathcal{L}} $ is called \textit{output layer}. Similarly, we call the input of the function in (\ref{INTRODNN_COMP}) as \textit{input layer.}
	We will only consider compositions, that are endowed with the \textit{rectified linear unit} (ReLU) activation function, i.e., for all $ d\in\N $ we have $ \mathbf{A}_d\colon\R^d\rightarrow\R^d $ satisfying for all $ x=(x_1,...,x_d)\in\R^d $ that $ \mathbf{A}_d(x)=(\max\{0,x_1\},...,\max\{0,x_d\}). $ 
	In this case, the function in (\ref{INTRODNN_COMP}) is characterized by the affine-linear mappings involved, which in turn are of the form $ 	T_k(x)=W_kx+B_k $, where $ \left( W_k,B_k \right)\in\left(\R^{l_k\times l_{k-1}}\times\R^{l_k}\right),x\in\R^{l_{k-1}}, $  $ k\in\{1,...,\mathcal{L}\}, l_0,...,l_\mathcal{L}\in\N, $  and thus the function in (\ref{INTRODNN_COMP}) is characterized by $ \left( W_1,B_1 \right),...,$  $\left( W_\mathcal{L},B_\mathcal{L}\right) $.
	We refer to the entries of the matrices $ W_k, k\in\{1,...,\mathcal{L}\}, $  as \textit{weights}, to the entries of the vectors $ B_k, k\in\{1,...,\mathcal{L}\}, $ as \textit{biases}, and to the collection $ \left(\left( W_1,B_1 \right),...,\left( W_\mathcal{L},B_\mathcal{L} \right)\right ) $ as \textit{neural network}.
	If $ \mathcal{L} \in \N\cap[2,\infty), $ we call the network $ \textit{deep}. $
	The set of all deep neural networks in our consideration can thus be defined as 
	\begin{equation}
		\mathbf{N}=\bigcup_{H\in\N}\bigcup_{(k_0,...,k_{H+1})\in\N^{H+2}}\left[\prod_{n=1}^{H+1}\left(\R^{k_n\times k_{n-1}}\times\R^{k_n}\right)\right].
	\end{equation}
	For a given $ \Phi \in \mathbf{N}, $ we refer to its corresponding function with $ \calR(\Phi). $\\
	
	\noindent In practice DNNs are employed in a host of machine learning related applications.
 In supervised learning the goal is to choose a $ \Phi\in\mathbf{N} $ in such a way, that for a given set of datapoints $ (x_p,y_p)_{p\in[1,P]\cap\N}:[1,P]\cap\N\rightarrow\R^d\times\R^m, $ $ d,m,P\in\N, $ we have for all $ p\in[1,P]\cap\N $ that $ (\calR(\Phi))(x_p)\approx y_p $. 
	Typically, one fixes the shape for a network and then aims to minimize a loss function, usually the sum of squared errors, with respect to the choice of weights and biases. Due to the large number of parameters involved in even small networks, this minimization problem is of high dimension and is additionally highly non-convex. Therefore, classical minimization techniques such as gradient descent or Newton's method are not suitable. To overcome this issue, stochastic gradient descent methods are employed. These seem to work well, although there exist relatively few theoretical results which prove convergence of stochastic gradient descent methods, see, e.g., \cite{fehrman2019convergence}. For an overview on deep learning-based algorithms for PDE approximations, see \cite{beck2021overview}.
	\subsection{A mathematical framework for deep neural networks}
	\begin{setting}\label{Setting31}
		Let $ \nrm{\cdot}, \mnrm{\cdot}\colon(\cup_{d\in\N}\R^d)\rightarrow [0,\infty) $ and $ \dim\colon(\cup_{d\in\N}\R^d)\rightarrow \N $ satisfy for all $ d\in\N $, $ x=(x_1,...,x_d)\in\R^d $ that $ \nrm{x} =\sqrt{\sum_{i=1}^d (x_i)^2} $, $ \mnrm{x}= \max_{i\in [1,d]\cap\N}|x_i| $, and $ \dim(x)=d $, let $ \mathbf{A}_d: \R^d\rightarrow\R^d $, $ d\in \N $, satisfy for all $ d\in\N $, $ x=(x_1,...,x_d) \in\R^d$ that
		\begin{equation}
			\mathbf{A}_d(x)=(\max\{x_1,0\},...,\max\{x_d,0\}),
		\end{equation}
		let $ \mathbf{D}=\cup_{H\in\N}\N^{H+2} $, let 
		\begin{equation}
			\mathbf{N} = \bigcup_{H\in\N}\bigcup_{(k_0,...,k_{H+1})\in\N^{H+2}}\left[\prod_{n=1}^{H+1}(\R^{k_n\times k_{n-1}}\times\R^{k_n})\right],
		\end{equation}
		let $ \mathcal{D}\colon\mathbf{N}\rightarrow\mathbf{D},\calP\colon\mathbf{N}\rightarrow\N, $ and $ \mathcal{R}\colon \mathbf{N}\rightarrow (\cup_{k,l\in\N}C(\R^k,\R^l)) $ satisfy for all $ H\in\N, k_0, k_1,...,k_H,k_{H+1}\in\N, \Phi = ((W_1,B_1),...,(W_{H+1},B_{H+1}))\in \left[\prod_{n=1}^{H+1}(\R^{k_n\times k_{n-1}}\times\R^{k_n})\right], x_0\in\R^{k_0},...,x_H\in\R^{k_H}$ with $ \forall n\in\N\cap[1,H]\colon x_n = \mathbf{A}_{k_n}(W_nx_{n-1}+B_n) $ that 
		\begin{equation}
			\calP(\Phi)=\sum_{n=1}^{H+1}k_n(k_{n-1}+1), \qquad	\mathcal{D}(\Phi)=(k_0,k_1,...,k_H,k_{H+1}), 
		\end{equation}
		\begin{equation}
			\mathcal{R}(\Phi)\in C(\R^{k_0},\R^{k_{H+1}})\quad\text{and}\quad (\mathcal{R}(\Phi))(x_0)=W_{H+1}x_H+B_{H+1}.
		\end{equation}
	\end{setting}
	\begin{setting}\label{Setting3}	
		Assume Setting \ref{Setting31}, let $ \odot \colon\mathbf{D}\times \mathbf{D}\rightarrow\mathbf{D}$ satisfy for all $ H_1, H_2 \in \N,$  $ \alpha =(\alpha_0,\alpha_1,...,\alpha_{H_1},\alpha_{H_1+1})\in \N^{H_1+2},$  $ \beta=$  $(\beta_0,\beta_1,...,\beta_{H_2},$  $\beta_{H_2+1})$  $\in\N^{H_2+2} $ that
		\begin{equation}\label{defodot}
			\alpha \odot \beta = (\beta_0,\beta_1,...,\beta_{H_2},\beta_{H_2+1}+\alpha_0,\alpha_1,...,\alpha_{H_1+1})\in\N^{H_1+H_2+3},
		\end{equation}
		let $ \boxplus\colon \mathbf{D}\times\mathbf{D}\rightarrow\mathbf{D}  $ satisfy for all $ H\in\N, \alpha =(\alpha_0,\alpha_1,...,\alpha_{H},\alpha_{H+1})\in \N^{H+2}, \beta=(\beta_0,\beta_1,...,\beta_{H},\beta_{H+1})\in\N^{H+2} $ that 
		\begin{equation}
			\alpha \boxplus \beta = (\alpha_0, \alpha_1+\beta_1,...,\alpha_H+\beta_H,\beta_{H+1}),
		\end{equation}
		let $ \mathfrak{n}_n^d\in\mathbf{D}, n\in[3,\infty)\cap\N, d\in\N  $, satisfy for all $ n\in[3,\infty)\cap\N, d\in \N$ that 
		\begin{equation}
			\mathfrak{n}_n^d=(d,\underbrace{2d,.........,2d}_{(n-2)\text{-times}},d)\in\N^n,
		\end{equation}
		and let $ \mathfrak{n}_n \in \mathbf{D}, n\in [3,\infty)\cap\N, $ satisfy for all $ n\in [3,\infty)\cap\N $ that 
		\begin{equation}
			\mathfrak{n}_n=\mathfrak{n}_n^1.
		\end{equation}
	\end{setting}
	\subsection{Properties of the proposed deep neural networks}
	For the proof of our main result in this section, Lemma \ref{U=R}, we need several auxiliary lemmas, which are presented in this section. The proofs of Lemmas \ref{triangle}-\ref{sumsDNN} appear  in \cite{Hutzenthaler_2020} and are therefore omitted.
	
	\begin{lemma}\label{triangle}
		([\citen{Hutzenthaler_2020}, Lemma 3.5])
	  Assume Setting \ref{Setting3}, let $ H,k,l\in\N,  $ and let $ \alpha,\beta \in \{k\}\times\N^H\times\{l\}. $ Then it holds that $ \mnrm{\alpha\boxplus\beta}\le\mnrm{\alpha}+\mnrm{\beta}. $ 
	\end{lemma}
	\begin{lemma}\label{composition}
		([\citen{Hutzenthaler_2020}, Lemma 3.8])
		Assume Setting \ref{Setting3} and let $ d_1,d_2,d_3\in\N, f\in C(\R^{d_2},\R^{d_3}), g\in C(\R^{d_1},\R^{d_2}), \alpha,\beta \in \mathbf{D} $ satisfy that $ f\in\mathcal{R}(\{\Phi\in\mathbf{N}:\mathcal{D}(\Phi)=\alpha\}) $ and $ g \in\mathcal{R}(\{\Phi\in\mathbf{N}:\mathcal{D}(\Phi)=\beta\}). $ Then it holds that $ (f\circ g) \in\mathcal{R}(\{\Phi\in\mathbf{N}:\mathcal{D}(\Phi)=\alpha \odot \beta\}) $. 
	\end{lemma}
	\begin{lemma}\label{sumsDNN}
		([\citen{Hutzenthaler_2020}, Lemma 3.9])
		Assume Setting \ref{Setting3} and let $ M,H,p,q\in \N, h_1,h_2,...,h_M\in\R, k_i\in\mathbf{D}, f_i\in C(\R^p,\R^q), i\in [1,M]\cap\N, $ satisfy for all $ i\in[1,M]\cap\N $ that $ 	\dim(k_i)= H+2 \quad and \quad f_i\in\calR\left(\{\Phi\in \mathbf{N}:\mathcal{D}(\Phi)=k_i\}\right). $
		Then it holds that
		\begin{equation}
			\sum_{i=1}^{M}h_if_i \in \calR\left(\left\{\Phi \in \mathbf{N}:\mathcal{D}(\Phi)=\boxplus_{i=1}^{M}k_i\right\}\right).
		\end{equation}
	\end{lemma}
	\noindent The following lemma extends  [\citen{jentzen2019proof}, Lemma 5.4] as well as [\citen{Hutzenthaler_2020}, Lemma 3.6], and proves the existence of a DNN of arbitrary length representing the identity function.
	\begin{lemma}\label{Identity}
		Assume Setting \ref{Setting3} and let $ d,H\in\N$. Then it holds  that 
		\begin{equation}
			\textup{Id}_{\R^d}\in \mathcal{R}(\{\Phi\in\mathbf{N}:\mathcal{D}(\Phi)=\mathfrak{n}_{H+2}^d\}).
		\end{equation}
	\end{lemma}
	\begin{proof}
		Let $ w_1=\binom{1}{-1} $ and  let $ W_1 \in \R^{2d\times d} $, $ B_1\in \R^{2d} $ satisfy 
		\begin{equation}
			W_1=
			\begin{pmatrix}
				w_1 & 0 &\ldots&0 \\
				0 & w_1&\ldots&0\\
				\vdots&\vdots &\ddots&\vdots\\
				0 &0 &\ldots&w_1
			\end{pmatrix}
			\in \R^{2d\times d},\quad B_1=
			0_{\R^{2d}}
			\in\R^{2d},
		\end{equation}
		for all $  l\in[2,H]\cap\N $ let $ W_l\in \R^{2d\times2d}, B_l\in \R^{2d} $ satisfy 
		\begin{equation}
			W_l=
			\textup{Id}_{\R^{2d}}
			\in \R^{2d\times 2d},\quad B_l= 0_{\R^{2d}}
			\in\R^{2d},
		\end{equation}
		and  let $ W_{H+1}\in \R^{d\times 2d}, B_{H+1}\in\R^d $ satisfy 
		\begin{equation}
			W_{H+1}= W_1^\top
			\in \R^{d\times 2d},\quad B_{H+1}= 0_{\R^d}
			\in\R^{d},
		\end{equation}
		 let $ \phi\in \mathbf{N} $ satisfy that $ \phi=((W_1,B_1),(W_2,B_2),...,(W_H,B_H),(W_{H+1},B_{H+1})), $ for every $ a\in \R  $ let $ a^+=\max\{a,0\} $, and let $ x^{0} =(x^{0}_1,...,x^{0}_d)\in \R^d $, $ x^{1},x^{2},...,x^{H} \in \R^{2d} $ satisfy for all $ n\in\N\cap[1,H] $ that 
		\begin{equation}
			x^{n}=\mathbf{A}_{2d}(W_nx^{n-1}+B_n).
		\end{equation}
		This implies that $ \calD(\phi)=\mathfrak{n}^d_{H+2}$.
		Furthermore,  it holds  that
		\begin{equation}
			\begin{aligned}
				x^1&=\mathbf{A}_{2d}(W_1x^0+B_1)=\mathbf{A}_{2d}\left  ((x_1^0,-x_1^0,...,x_d^0,-x_d^0)^\top\right )=\left ((x_1^0)^+,(-x_1^0)^+,...,(x_d^0)^+,(-x_d^0)^+\right )^\top,\\
				x^2&=\mathbf{A}_{2d}(W_2x^1+B_2)=\mathbf{A}_{2d}\left (\left ((x_1^0)^+,(-x_1^0)^+,...,(x_d^0)^+,(-x_d^0)^+\right )^\top\right )\\
				&=\left ((x_1^0)^+,(-x_1^0)^+,...,(x_d^0)^+,(-x_d^0)^+\right )^\top,\\
			\end{aligned}
		\end{equation}
	and continuing this procedure yields
	\begin{equation}
		\begin{aligned}
			x^H&=\mathbf{A}_{2d}(W_Hx^{H-1}+B_H)=\mathbf{A}_{2d}\left (\left ((x_1^0)^+,(-x_1^0)^+,...,(x_d^0)^+,(-x_d^0)^+\right )^\top\right )\\
			&=\left ((x_1^0)^+,(-x_1^0)^+,...,(x_d^0)^+,(-x_d^0)^+\right )^\top.
		\end{aligned}
	\end{equation}
It thus holds that 
		\begin{equation}
			\left(\calR(\phi)\right)(x^0)=W_{H+1}x^H+B_{H+1}=x^0.
		\end{equation}
		This shows, since  $ x^{0}\in \R^d $ was chosen arbitrarily,  that $ \calR(\phi)=\textup{Id}_{\R^d}. $ This and the fact that it holds that $ \calD(\phi) =\mathfrak{n}_{H+2}^d$ prove that $ \textup{Id}_{\R^d}\in \calR\left(\left\{\Phi\in\mathbf{N}:\calD(\Phi)=\mathfrak{n}_{H+2}^d\right\}\right). $ The proof is thus completed.
	\end{proof}
	
	\noindent The following auxiliary lemma states that if a given function can be represented by a DNN of  certain depth, then it can be represented by a DNN of any larger depth.
	\begin{lemma}\label{prolonging}
		Assume Setting \ref{Setting3}, let $ H,p,q\in\N$ and let $ g\in C(\R^p,\R^q) $ satisfy that $ g$  $ \in $ \\ $\calR \left(\{\Phi\in\mathbf{N}\colon\dim(\calD(\Phi))=H+2\}\right) $. Then for all $ n\in\N_0 $ we have 
		\begin{equation}
			g\in\calR\left(\{\Phi\in\mathbf{N}\colon \dim(\calD(\Phi))=H+2+n\}\right).
		\end{equation}
	\end{lemma}
	\begin{proof}
		Let $ k_0, k_1,...,k_H,k_{H+1}\in\N, \Phi_g =( (W_1, B_1),...,(W_{H+1},B_{H+1})) \in \left[\prod_{n=1}^{H+1}(\R^{k_n\times k_{n-1}}\times \R^{k_n})\right]$ satisfy for all $ x_0\in\R^{k_0},...,x_H\in\R^{k_H} $ with $ \forall n\in\N\cap[1,H]\colon x_n=\mathbf{A}_{k_n}(W_nx_{n-1}+B_n) $ that $ (\calR(\Phi_g))(x_0)=W_{H+1}x_H+B_{H+1}=g(x_0). $
		For $ n=0, $ the claim is clear.
		For the case that $ n\in\N\cap[2,\infty), $ Lemma \ref{Identity} ( applied with $ H\curvearrowleft n-1 $ in the notation of Lemma \ref{Identity}) proves that $ \textup{Id}_{\R^q} \in \calR(\{\Phi\in\mathbf{N}\colon\calD(\Phi)=\mathfrak{n}^q_{n+1}\})$. This and Lemma \ref{composition} (applied with $ d_1\curvearrowleft p, d_2 \curvearrowleft q, d_3 \curvearrowleft q , g\curvearrowleft g, f \curvearrowleft \textup{Id}_{\R^q}, \alpha \curvearrowleft \mathfrak{n}^q_{n+1} ,\beta \curvearrowleft \calD(\Phi_g)$ in the notation of Lemma \ref{composition}) ensure that $ g\in\calR(\{\Phi\in\mathbf{N}\colon \dim(\calD(\Phi))=H+2+n\}). $\\
		For the case $ n=1, $  let $ \Psi=( (W_1, B_1),...,(W_H,B_H), (\textup{Id}_{\R^{k_H}},0_{\R^{k_H}}),(W_{H+1},B_{H+1})).  $
		This and the fact that for all $ x\in\R $ it holds that $ \max\{0,x\}=\max\{0,\max\{0,x\}\} $ ensures for all $ x_0\in\R^{k_0},...,x_H\in\R^{k_H} $ with $ \forall n\in\N\cap[1,H]\colon x_n=\mathbf{A}_{k_n}(W_nx_{n-1}+B_n) $ that 
		\begin{equation}
			\begin{aligned}
				(\calR(\Phi_g))(x_0)&=W_{H+1}\mathbf{A}_{k_H}(W_Hx_{H-1}+B_H) + B_{H+1}\\
				&=W_{H+1}\mathbf{A}_{k_H}(\mathbf{A}_{k_H}(W_Hx_{H-1}+B_H)) +B_{H+1}\\
				&=W_{H+1}\mathbf{A}_{k_H}(0_{\R^{k_H}}+\textup{Id}_{\R^{k_H}}\mathbf{A}_{k_H}(W_Hx_{H-1}+B_H))+B_{H+1}\\
				&=(\calR(\Psi))(x_0).
			\end{aligned}
		\end{equation}
		This combined with the fact that $ \dim(\calD(\Psi))=\dim(\calD(\Phi_g))+1$ completes the  proof.
	\end{proof}
	\subsection{A DNN representation for time-recursive functions}
	\begin{lemma}\label{DNN-recursion}
		Assume Setting \ref{Setting3}, let $ d,m,K\in \N, T\in (0,\infty),  $ let $ \tau_0,\tau_1,...\tau_K\in[0,T] $ satisfy $ 0\le \tau_0\le ...\le \tau_K\le T, $ let $ \sigma \in C(\R^d,\R^{d\times m}), $ for all $ v\in \R^d $ let $ \Phi_{\sigma,v}\in \mathbf{N} $ satisfy that $ \calR(\Phi_{\sigma,v})  = \sigma(\cdot)\cdot v$ and $ \calD(\Phi_{\sigma,v})=\calD(\Phi_{\sigma,0}) ,$ let $ \llcorner \cdot \lrcorner \colon \R\rightarrow\R $ satisfy for all $ t\in \R $ that $ \llcorner t \lrcorner =\max(\{\tau_0,\tau_1,...,\tau_K\}\cap ((-\infty,t)\cup\{\tau_0\})), $  let $ f\colon \R \rightarrow\R^m $, 
		for all $ s\in[0,T] $ let $ g_s\colon \R^d\rightarrow\R^d $ satisfy for all $ x\in\R^d$ that $ g_0(x)=x,  $ 
		\begin{equation}
			g_s(x)=g_{\llcorner s\lrcorner}(x) + \sigma(g_{\llcorner s\lrcorner}(x))(f(s)-f(\llcorner s\lrcorner)).
		\end{equation}
		Then for all $ s\in [0,T] $ there exists a $ \frakg_s \in \mathbf{N} $ such that 
		\begin{enumerate}[(i)]
			\item $ \calR(\frakg_s) =g_s$,
			\item $ \calD(\frakg_s) = \underset{l=1}{\overset{K}{\odot}} \left[\mathfrak{n}^d_{\dim(\calD(\Phi_{\sigma},0))}\boxplus \calD(\Phi_{\sigma,0})\right]$.
		\end{enumerate}
		
	\end{lemma}
	\begin{proof}
		Note that for all $ s\in [0,T], k\in [1,K]\cap\N $ we have 
		\begin{equation}
			(s\vee \tau_{k-1})\wedge \tau_k = \begin{cases}
				\tau_{k-1} &, s\le \tau_{k-1}\\
				s &, s\in (\tau_{k-1},\tau_k]\\
				\tau_k &, s>\tau_k
			\end{cases}.
		\end{equation}
		This ensures  for all $ s\in[0,T], x\in \R^d $ that
		\begin{equation}
			g_s(x)=x+\sum_{k=1}^K\sigma(g_{\tau_{k-1}}(x))\left[f((s\vee \tau_{k-1})\wedge \tau_k)- f(\tau_{k-1}) \right].
		\end{equation}
		Next for all $ s\in[0,T], k\in [1,K]\cap\N $ let $ \Phi_s^k \in C(\R^d, \R^d) $ satisfy for all $ x\in \R^d $ that 
		\begin{equation}
			\Phi_s^k(x)=x+\sigma(x)\left[f((s\vee \tau_{k-1})\wedge \tau_k)- f(\tau_{k-1})\right].
		\end{equation}
		For all $ k\in [1,K]\cap\N, s \in [0,T] $ let $ \Psi_s^k \in C(\R^d,\R^d) $ satisfy 
		\begin{equation}
			\Psi_s^k= \Phi_s^k \circ \Phi_s^{k-1}\circ \dots \circ \Phi_s^1.
		\end{equation}
	Note that for all $ k\in [1,K-1]\cap\N, s\le \tau_{k-1} $ it holds $ \Phi_s^k=\textup{Id}_{\R^d}. $
		This ensures for all $ k\in [1,K-1]\cap\N, n\in [k+1,K]\cap\N, s\le \tau_k $ that 
		\begin{equation}
			\Psi_s^k=\Psi_s^n,
		\end{equation}
		and in particular $ \Psi_s^k=\Psi_s^K. $
		Observe that for all $ k\in [1,K]\cap\N, s\le \tau_k, x\in\R^d $ we have $ \Psi_s^k(x)=g_s(x) $ and thus, for all $ s\in[0,T], x\in \R^d $ it holds $ \Psi_s^K(x)=g_s(x). $\\
		
		\noindent Next, Lemma \ref{Identity} (applied with $ d\curvearrowleft d, H \curvearrowleft \dim(\calD(\Phi_{\sigma,0}))-2$ in the notation of Lemma \ref{Identity}) ensures for all $ s\in[0,T] $ that 
		\begin{equation}
			\textup{Id}_{\R^d}\in \calR\left(\left\{\Phi\in\mathbf{N}\colon \calD(\Phi)=\mathfrak{n}^d_{\dim(\calD(\Phi_{\sigma,0}))}\right\}\right).
		\end{equation}
		Note that Lemma \ref{sumsDNN} (applied for all $ s\in[0,T], k\in [1,K]\cap\N $ with $ M\curvearrowleft 2, p\curvearrowleft d, q\curvearrowleft d, H \curvearrowleft\dim(\calD(\Phi_{\sigma,0})), h_1\curvearrowleft 1, h_2 \curvearrowleft1, f_1\curvearrowleft\textup{Id}_{\R^d}, f_2 \curvearrowleft \sigma(\cdot)\left[f((s\vee\tau_{k-1})\wedge\tau_k)-f(\tau_{k-1})\right] $, $ k_1\curvearrowleft \mathfrak{n}^d_{\dim(\calD(\Phi_{\sigma,0}))}, k_2 \curvearrowleft \calD(\Phi_{\sigma,0})$ in the notation of Lemma \ref{sumsDNN}) proves for all $ s\in [0,T], k\in[1,K]\cap\N $ that 
		\begin{equation}
			\Phi_s^k\in \calR\left(\left\{\Phi\in\mathbf{N}\colon \calD(\Phi)= \mathfrak{n}^d_{\dim(\calD(\Phi_{\sigma,0}))}\boxplus \calD(\Phi_{\sigma,0})  \right\}\right).
		\end{equation}
		This and Lemma \ref{composition} (applied for all $ s\in[0,T] $ with $ d_1\curvearrowleft d, d_2\curvearrowleft d, d_3\curvearrowleft d,  $ $ f \curvearrowleft \Phi_s^2, g\curvearrowleft \Phi_s^1 $, $ \alpha \curvearrowleft \mathfrak{n}^d_{\dim(\calD(\Phi_{\sigma,0}))}\boxplus \calD(\Phi_{\sigma,0}) , \beta \curvearrowleft \mathfrak{n}^d_{\dim(\calD(\Phi_{\sigma,0}))}\boxplus \calD(\Phi_{\sigma,0})  $ in the notation of Lemma \ref{composition}) show for all $ s\in[0,T] $ that 
		\begin{equation}
			\Psi_{s}^2=\Phi_s^2 \circ \Phi_s^1 \in \calR\left(\left\{ \Phi\in \mathbf{N}\colon \calD(\Phi)= \underset{l=1}{\overset{2}{\odot}}\left[\mathfrak{n}^d_{\dim(\calD(\Phi_{\sigma,0}))}\boxplus \calD(\Phi_{\sigma,0}) \right] \right\}\right).
		\end{equation}
		This and Lemma \ref{composition} (applied for all $ s\in[0,T] $ with $ d_1\curvearrowleft d, d_2\curvearrowleft d, d_3\curvearrowleft d,  $ $ f \curvearrowleft \Phi_s^3, g\curvearrowleft \Psi_s^2 $, $ \alpha \curvearrowleft \mathfrak{n}^d_{\dim(\calD(\Phi_{\sigma,0}))}\boxplus \calD(\Phi_{\sigma,0}) , \beta \curvearrowleft  \underset{l=1}{\overset{2}{\odot}}\left[\mathfrak{n}^d_{\dim(\calD(\Phi_{\sigma,0}))}\boxplus \calD(\Phi_{\sigma,0}) \right]  $ in the notation of Lemma \ref{composition}) show for all $ s\in[0,T] $ that 
		\begin{equation}
			\Psi_{s}^3=\Phi_s^3 \circ \Psi_s^2 \in \calR\left(\left\{ \Phi\in \mathbf{N}\colon \calD(\Phi)= \underset{l=1}{\overset{3}{\odot}}\left[\mathfrak{n}^d_{\dim(\calD(\Phi_{\sigma,0}))}\boxplus \calD(\Phi_{\sigma,0}) \right] \right\}\right).
		\end{equation}
		Continuing this procedure hence demonstrates that for all $ s\in[0,T], k\in[1,K]\cap\N $ it holds 
		\begin{equation}
			\Psi_{s}^k\in \calR\left(\left\{ \Phi\in \mathbf{N}\colon \calD(\Phi)= \underset{l=1}{\overset{k}{\odot}}\left[\mathfrak{n}^d_{\dim(\calD(\Phi_{\sigma,0}))}\boxplus \calD(\Phi_{\sigma,0}) \right] \right\}\right).
		\end{equation}
		This and the fact that for all $ s\in[0,T], x\in\R^d $ it holds $ \Psi_s^K(x)=g_s(x) $ establishes for all $ s\in[0,T] $ the existence of $ \frakg_s \in \mathbf{N} $ satisfying for all $ x\in\R^d $ that $ \calR(\frakg_s)\in C(\R^d,\R^d) $ and 
		\begin{equation}
			\calD(\frakg_s)=\underset{l=1}{\overset{K}{\odot}}\left[\mathfrak{n}^d_{\dim(\calD(\Phi_{\sigma,0}))}\boxplus \calD(\Phi_{\sigma,0}) \right].
		\end{equation}
		This finishes the proof.
	\end{proof}
	\subsection{A DNN representation of Euler-Maruyama approximations}
	In the following lemma we construct a DNN that represents an Euler-Maruyama approximation. 
	This construction is similar to the one developed in [\citen{grohs2019spacetime}], where a DNN representation of Euler approximations for ordinary differential equations is constructed.
	
	\begin{lemma}\label{DNNEULER}
		Assume Setting \ref{Setting3}, let $ d,K\in\N, T\in (0,\infty)$, $ \tau_0,...,\tau_K\in[0,T] $ satisfy $ 0\le \tau_0 \le ... \le \tau_K\le T, $ let $ \mu\in C(\R^d,\R^d), \sigma \in C(\R^d,\R^{d\times d}), $ for all $ v\in\R^d $ let $ \Phi_\mu, \Phi_{\sigma,v}\in\mathbf{N} $ satisfy $ \calR(\Phi_\mu)=\mu, \calR(\Phi_{\sigma,v})=\sigma(\cdot)v, $  $ \calD(\Phi_{\sigma,v})=\calD(\Phi_{\sigma,0}), $
		let $ (\Omega, \mathcal{F}, \mathbb{P}, (\F_t)_{t\in[0,T]}) $ be a filtered probability space satisfying the usual conditions, let $ \Theta = \cup_{n \in \N}\mathbb{Z}^n, $ let $ \mathbf{W}^{\theta}\colon \Omega\times [0,T]\rightarrow \R^d, \theta\in \Theta, $ be i.i.d. standard $ (\F_t)_{t\in[0,T]}$-Brownian motions, let $ \llcorner \cdot \lrcorner\colon \R\rightarrow\R $ satisfy for all $ t\in \R $ that $ \llcorner t\lrcorner = \max(\{\tau_0, ...\tau_K\}\cap ((-\infty,t)\cup\{\tau_0\})), $ for all $ t\in [0,T], x\in \R^d, \theta \in \Theta $ let $ Y_{t}^{\theta,x}\colon [t,T]\times \Omega \rightarrow \R^d,  $ be measurable and satisfy for all $ s\in[t,T]$ that $ Y_{t,t}^{\theta,x}=x $ and 
		\begin{equation}
			Y_{t,s}^{\theta,x}-Y_{t,\max\{t,\llcorner s\lrcorner\}}^{\theta,x} = \mu(Y_{t,\max\{t,\llcorner s\lrcorner\}}^{\theta,x})(s-\max\{t,\llcorner s\lrcorner\}) +\sigma(Y_{t,\max\{t,\llcorner s\lrcorner\}}^{\theta,x})(\mathbf{W}_{s}^\theta-\mathbf{W}^\theta_{\max\{t,\llcorner s \lrcorner\}}),
		\end{equation}
		and let $ \omega\in\Omega. $ Then there exists a family $ (\calY_{t,s}^\theta)_{t,s\in[0,T], \theta\in\Theta}\subseteq \mathbf{N} $ such that 
		\begin{enumerate}[(i)]
			\item for all $ t\in[0,T], s\in[t,T],\theta\in\Theta $ it holds that $ \calR(\calY_{t,s}^\theta)\in C(\R^d,\R^d) $ and $ (\calR(\calY_{t,s}^\theta))(x)=Y_{t,s}^{\theta,x}(\omega), $
			\item for all $ t_1,t_2\in[0,T], $ $ s_1\in[t_1,T],s_2\in[t_2,T] ,\theta_1,\theta_2\in\Theta$ it holds that $ \calD(\calY_{t_1,s_1}^{\theta_1})=\calD(\calY_{t_2,s_2}^{\theta_2}), $
			\item for all $ t\in[0,T], s\in[t,T], \theta\in\Theta $ it holds that $ \dim(\calD(\calY_{t,s}^\theta))=K(\max\{\dim(\calD(\Phi_\mu)), \dim(\calD(\Phi_{\sigma,0}))\}-1), $
			\item for all $ t\in[0,T], s\in[t,T], \theta\in\Theta $ it holds that $ \mnrm{\calD(\calY_{t,s}^\theta)}\le \max\{2d, \mnrm{\calD(\Phi_\mu)}, \mnrm{\calD(\Phi_{\sigma,0})}\}. $
		\end{enumerate}
	\end{lemma}
	\begin{proof}
		Without loss of generality assume that $ \dim(\Phi_\mu)=\dim(\calD(\Phi_{\sigma,0})), $ otherwise we may refer to Lemma \ref{prolonging}.
		Let $ \Sigma \in C(\R^d,\R^{d\times(d+1)}) $ satisfy for all $ x\in\R^d, i\in \{1,...,d\}, j\in \{1,...,d+1\} $ that 
		\begin{equation}
			(\Sigma(x))_{i,j}=\begin{cases}
				\mu_i(x) &, j=1\\
				\sigma_{i,j-1}(x)&,\text{ else }
			\end{cases}.
		\end{equation}
	Hence, for all $ v=(v_1,v')\in\R\times \R^d, $ $ x\in\R^d $ it holds that
	\begin{equation}
		\Sigma(x)v= \mu(x)v_1+\sigma(x)v'.
	\end{equation}
This and Lemma \ref{sumsDNN} (applied for all $ t\in[0,T], \theta\in\Theta, v=(v_1,v')\in\R\times\R^d $ with $ M\curvearrowleft 2 $, $ H\curvearrowleft\dim(\calD(\Phi_{\sigma,0})) $, $ p\curvearrowleft d, $ $ q\curvearrowleft d, $ $ h\curvearrowleft v_1, $ $ h_2\curvearrowleft 1 $, $ k_1\curvearrowleft\calD(\Phi_\mu) $, $ k_2\curvearrowleft\calD(\Phi_{\sigma,0}) ,$ $ f_1\curvearrowleft\mu, $ $ f_2\curvearrowleft\sigma(\cdot)v' $ in the notation of Lemma \ref{sumsDNN}) ensure for all $ t\in[0,T],\theta\in\Theta, v=(v_1,v')\in\R\times\R^d $ that
\begin{equation}
	\Sigma(\cdot)v \in \calR\left( \{\Phi \in \mathbf{N}\colon \calD(\Phi)=\calD(\Phi_\mu)\boxplus\calD(\Phi_{\sigma,0})\}\right).
\end{equation}
		For all $ \theta \in \Theta $ let $ f^\theta\colon \R \rightarrow \R^{d+1} $ satisfy for all $ t\in [0,T] $ that 
		\begin{equation}
			f^\theta(t)=\begin{pmatrix}
				t\\
				\mathbf{W}^\theta_t(\omega)
			\end{pmatrix}.
		\end{equation}
		Observe that for all $ x\in\R^d, t\in [0,T], \theta \in \Theta $ it holds
		\begin{equation}
			\Sigma(x)f^\theta(t)=\mu(x)t +\sigma(x)\mathbf{W}^\theta_t(\omega).
		\end{equation}
		Hence Lemma \ref{DNN-recursion} (applied for all $ t\in[0,T], s\in[t,T] \theta \in \Theta $ with $ d \curvearrowleft d, m\curvearrowleft d+1, $ $ K\curvearrowleft K, $  $ \tau_0 \curvearrowleft \tau_0 \vee t, $  $ \tau_1 \curvearrowleft \tau_1 \vee t,$...$, \tau_K \curvearrowleft \tau_K \vee t,  $ $ \sigma \curvearrowleft \Sigma, $ $ f\curvearrowleft f^\theta $, $ g_s \curvearrowleft Y_{t,s}^{\theta, \cdot} $ in the notation of Lemma \ref{DNN-recursion}) shows for all $ t\in[0,T], s\in[t,T], \theta \in \Theta $ the existence of a $ \calY_{t,s}^\theta \in \mathbf{N} $ such that 
		\begin{enumerate}[(I)]
			\item $ \calR(\calY_{t,s}^{\theta})\in C(\R^d,\R^d), $ satisfying for all $ x\in \R^d $ that $ (\calR(\calY_{t,s}^{\theta}))(x)=Y_{t,s}^{\theta,x}(\omega), $
			\item $ \calD(\calY_{t,s}^\theta)= \underset{l=1}{\overset{K}{\odot}}\left[\mathfrak{n}^d_{\dim(\calD(\Phi_{\sigma,0}))}\boxplus\calD(\Phi_\mu)\boxplus \calD(\Phi_{\sigma,0})\right]. $
		\end{enumerate}
		Observe that $ \calD(\calY_{t,s}^\theta) $ does not depend on $ t,s $ or $ \theta,$ which thus implies item $ (ii) $. 
		Item $ (iii) $ follows from item (II) and the definition of $ \odot. $
		
		Note that for all $ H_1, H_2, \alpha_0,...,\alpha_{H_1+1}, \beta_0,...,\beta_{H_2+1} \in \N $, $ \alpha, \beta \in \mathbf{D} $ satisfying $ \alpha=(\alpha_0,...,\alpha_{H_1+1}) $ and $ \beta=(\beta_0,...,\beta_{H_2+1}) $ with $ \alpha_0=\beta_{H_2+1} $ it holds that $ \mnrm{\alpha \odot \beta}\le \max\{\mnrm{\alpha},\mnrm{\beta}, 2\alpha_0\}. $
		This combined with item (II) imply item $ (iv). $
		The proof is thus completed.
	\end{proof}
	\color{black}
	\subsection{A DNN representation of MLP approximations}
	In the following lemma we establish a central result of this article: MLP approximations can be represented by DNNs. With the help of Lemma \ref{DNNEULER} we extend [\citen{Hutzenthaler_2020}, Lemma 3.10], where the special case of MLP approximations for  semilinear heat equations is treated.
	\begin{lemma}\label{U=R}
		Assume Setting \ref{Setting3}, let $ c,d,K,M\in \N $, $ T \in (0,\infty) $, let $ \tau_0,...,\tau_K\in [0,T] $ satisfy $ 0=\tau_0 \le \tau_1\le ...\le \tau_K=T $, let $ \mu \in C(\R^d,\R^d), $ $ \sigma =(\sigma_{i,j})_{i,j\in\{1,...,d\}}\in C(\R^d,\R^{d\times d}),  $ $ f\in C(\R,\R), $ $ g\in C(\R^d,\R) $, 
		for all $ v\in\N $  let $ \Phi_\mu,$  $ \Phi_{\sigma,v},$  $ \Phi_f, \Phi_g  \in \mathbf{N}$ satisfy $ \calR(\Phi_\mu)=\mu, \calR(\Phi_f)=f, \calR(\Phi_g)=g ,$ $ \calR(\Phi_{\sigma,v})=\sigma(\cdot)v, $
		$ \calD(\Phi_{\sigma,v})=\calD(\Phi_{\sigma,0}) $, 
		\begin{equation}
			c \ge \max\{ 2d, \mnrm{\calD(\Phi_f)},\mnrm{\calD(\Phi_g)}, \mnrm{\calD(\Phi_\mu)},\mnrm{\calD(\Phi_{\sigma,0})}\},
		\end{equation}
		let $ (\Omega,\mathcal{F},\mathbb{P},(\F_t)_{t\in[0,T]}) $ be a  filtered probability space satisfying the usual conditions, let $ \Theta = \cup_{n\in\N}\mathbb{Z}^n $, let $ \mathfrak{t}^\theta\colon\Omega\rightarrow[0,1] , \theta\in\Theta$, be i.i.d. random variables,
		assume for all $ t\in(0,1)$ that $ \mathbb{P}(\mathfrak{t}^0\le t)=t $, let $ \mathfrak{T}^\theta\colon[0,T]\times\Omega\rightarrow[0,T] $ satisfy for all $ \theta\in\Theta, t\in[0,T] $ that $ \mathfrak{T}_t^\theta =t+(T-t)\mathfrak{t}^\theta $,
		let $ W^\theta\colon[0,T]\times\Omega\rightarrow\R^d, \theta \in \Theta, $ be i.i.d. standard $ (\F_t)_{t\in[0,T]}$-Brownian motions, assume that $ (\mathfrak{t}^\theta)_{\theta\in\Theta} $ and $ (W^\theta)_{\theta \in \Theta} $ are independent,
		let $ \lfloor\cdot\rfloor\colon\R\rightarrow\R $ satisfy for all $ t\in\R $ that $ \lfloor t \rfloor= \max\left(\{\tau_0, \tau_1,...,\tau_K\}\cap ((-\infty,t)\cup\{\tau_0\})\right) $, 
		let $ Y_t^{\theta,x}\colon[t,T]  \times \Omega \rightarrow \R^d$, $ \theta\in\Theta,t\in[0,T], x\in\R^d$ be measurable, satisfying for all $ t\in [0,T], x\in\R^d, s\in[t,T],\theta\in\Theta, \omega\in\Omega $ that $ Y_{t,t}^{\theta,x}=x $ and
		\begin{equation}
			Y_{t,s}^{\theta,x}-Y_{t,\max\{t,\lfloor s\rfloor\}}^{\theta,x} = \mu(Y_{t,\max\{t,\lfloor s\rfloor\}}^{\theta,x})(s-\max\{t,\lfloor s\rfloor\})+\sigma(Y_{t,\max\{t,\lfloor s\rfloor\}}^{\theta,x})(W_s^\theta-W^\theta_{\max\{t,\lfloor s\rfloor\}})
		\end{equation}
		let $ U_n^\theta\colon[0,T]\times\R^d\times\Omega\rightarrow\R, n\in \mathbb{Z}, \theta\in\Theta,$ satisfy for all $ \theta \in \Theta, n\in\N_0, t\in [0,T], x\in\R^d $ that
		\begin{equation}\label{DNNMLP_MLP}
			\begin{aligned}
				&U_n^\theta(t,x)=\dfrac{\mathbbm{1}_\N(n)}{M^n}\sum_{i=1}^{M^n}g(Y_{t,T}^{(\theta,0,-i),x})\\
				&+\sum_{l=0}^{n-1}\dfrac{(T-t)}{M^{n-l}}\left[\sum_{i=1}^{M^{n-l}}(f\circ U_l^{(\theta,l,i)}-\mathbbm{1}_\N(l)f\circ U_{l-1}^{(\theta,-l,i)})(\mathfrak{T}_t^{(\theta,l,i)},Y_{t,\mathfrak{T}_t^{(\theta,l,i)}}^{(\theta,l,i),x})\right],
			\end{aligned}
		\end{equation}
		and let $ \omega \in \Omega$. Then for all $ n\in\N_0 $ there exists a family $ (\Phi_{n,t}^\theta)_{\theta\in\Theta, t\in[0,T]}\subseteq \mathbf{N}$ such that 
		\begin{enumerate}[(i)]
			\item\label{DNNMLP_item_i} it holds for all $ t_1,t_2\in[0,T], \theta_1,\theta_2\in\Theta $ that \begin{equation}
				\mathcal{D}(\Phi_{n,t_1}^{\theta_1})=\mathcal{D}(\Phi_{n,t_2}^{\theta_2})
			\end{equation}
			\item\label{DNNMLP_item_ii} it holds for all $ t\in[0,T],\theta\in\Theta $ that \begin{equation}
				\begin{aligned}
					\dim(\mathcal{D}(\Phi_{n,t}^\theta))&=(n+1)K(\max\{\dim(\calD(\Phi_\mu)),\dim(\calD(\Phi_{\sigma,0})))\}-1)-1\\
					&\quad+n\left(\dim(\mathcal{D}(\Phi_f))-2\right)+\dim(\mathcal{D}(\Phi_g)),
				\end{aligned}
			\end{equation}
			\item\label{DNNMLP_item_iii} it holds for all $ t\in[0,T],\theta\in\Theta $ that \begin{equation}
				\mnrm{\mathcal{D}(\Phi_{n,t}^\theta)}\le c(3M)^n,
			\end{equation} 
			\item\label{DNNMLP_item_iv} it holds for all $ t\in [0,T],\theta \in \Theta, x\in\R^d $ that \begin{equation}
				U_{n}^\theta(t,x,\omega)=(\mathcal{R}(\Phi_{n,t}^\theta))(x).
			\end{equation}
		\end{enumerate}
	\end{lemma}
	\begin{proof}
		Throughout this proof let $ \mathbb{Y}\in\N $ satisfy 
		\begin{equation}
			\mathbb{Y}=K(\max\{\dim(\calD(\Phi_\mu)),\dim(\calD(\Phi_{\sigma,0}))\}-1).
		\end{equation}
		We will prove this lemma via induction on $ n\in\N $. For $ n=0 $ we have for all $ \theta,\in\Theta, x\in\R^d, t\in[0,T] $ that $ U_0^\theta(t,x,\omega)=0 $. By choosing all weights and biases to be 0, the 0-function can be represented by deep neural networks of arbitrary shape. Therefore items $ (\ref{DNNMLP_item_i})$-$(\ref{DNNMLP_item_iv}) $ hold true for $ n=0 $.
		For the induction step $ n\rightarrow n+1 $ we will construct $ U_{n+1}^\theta(t,x), \theta\in\Theta,t\in[0,T],x\in\R^d, $ via DNN functions, which we achieve by representing each summand of (\ref{DNNMLP_MLP}) by a DNN, and then, in order to be able to utilize Lemma \ref{sumsDNN}, compose each DNN with a suitable sized network representing the identity function on $ \R $. 
		Lemma \ref{DNNEULER} yields the existence of a family $ (\calY_{t,s}^\theta)_{\theta\in\Theta,s,t\in[0,T]}\subseteq\mathbf{N}$ satisfying for all $\theta\in\Theta,t\in[0,T],s\in[t,T]  $ that $ (\calR(\calY_{t,s}^\theta))(x)= Y_{t,s}^{\theta,x}(\omega), $ $ \dim(\calD(\calY_{t,s}^\theta))=\mathbb{Y}, $ $ \mnrm{\calD(\calY_{t,s}^{\theta})}\le \max\{2d, \mnrm{\calD(\Phi_\mu)},\mnrm{\calD(\Phi_{\sigma,0})}\},  $ and $ \calD(\calY_{t,s}^\theta)=\calD(\calY_{0,0}^0) $.  Lemma \ref{composition} (applied for all $ \theta \in \Theta, x\in \R^d $ with $ d_1\curvearrowleft d, d_2 \curvearrowleft d, d_3 \curvearrowleft 1, \alpha \curvearrowleft \calD(\Phi_g), \beta \curvearrowleft \calD(\calY_{t,T}^\theta), f \curvearrowleft g, g\curvearrowleft Y_{t,T}^{\theta,x} $ in the notation of Lemma \ref{composition}) ensures for all $ \theta \in \Theta, x\in \R^d,  $ that 
		\begin{equation}\label{DNNMLPproof1}
			g\circ Y_{t,T}^{\theta,x} \in \calR\left(\left\{ \Phi \in \mathbf{N}: \calD(\Phi)=\calD(\Phi_g)\odot\calD(\calY_{t,T}^\theta)   \right\}\right).
		\end{equation} 
		Note that Lemma \ref{Identity} (applied with $ H \curvearrowleft (n+1)(\dim(\calD(\Phi_f))-2+\mathbb{Y})+1 $ in the notation of Lemma \ref{Identity}) ensures that 
		\begin{equation}
			\textup{Id}_{\R}\in \calR\left(\left\{\Phi\in\mathbf{N}:\calD(\Phi)=\mathfrak{n}_{(n+1)(\dim(\calD(\Phi_f))-2+\mathbb{Y})+1}\right\}\right).
		\end{equation}
		This, (\ref{DNNMLPproof1}), and Lemma \ref{composition} (applied for all $ t\in[0,T], \theta,\in\Theta,x\in\R^d $ with $ d_1\curvearrowleft d, d_2 \curvearrowleft 1, d_3 \curvearrowleft 1 , f=\textup{Id}_\R, g= g \circ Y_{t,T}^{\theta,x},$  $ \alpha = \mathfrak{n}_{(n+1)(\dim(\calD(\Phi_f))-2+\mathbb{Y})+1},$  $ \beta= \calD(\calY_{t,T}^\theta) $  in the notation of Lemma \ref{composition}) show that for all $ t\in[0,T], \theta\in\Theta, x\in\R^d $ it holds that
		\begin{equation}
			g\circ Y_{t,T}^{\theta,x} \in \calR\left(\left\{ \Phi \in \mathbf{N}: \calD(\Phi)=\mathfrak{n}_{(n+1)(\dim(\calD(\Phi_f))-2+\mathbb{Y})+1}\odot\calD(\Phi_g)\odot\calD(\calY_{t,T}^\theta)   \right\}\right).
		\end{equation}
		Next, the induction hypothesis implies for all $ l\in[0,n]\cap\N_0, t\in[0,T], \theta \in \Theta $ the existence of   $ \Phi_{l,t}^\theta\in\mathbf{N} $ satisfying  $\calR(\Phi_{l,t}^\theta)= U_l^\theta(t,\cdot,\omega) $ and $ \calD(\Phi_{l,t}^\theta)=\calD(\Phi_{l,0}^0). $ 
		This and Lemma \ref{composition} (applied for all $ \theta,\eta\in\Theta, t\in[0,T], l\in[0,n]\cap \N_0 $ with $ d_1\curvearrowleft d, d_2 \curvearrowleft d, d_3 \curvearrowleft 1 $, $ f \curvearrowleft U_n^\eta(\frakT^\theta_t(\omega),\cdot,\omega) $, $ g \curvearrowleft Y_{t,\frakT^\theta_t(\omega)}^{\theta,x}  $, $ \alpha\curvearrowleft \calD(\Phi_{l,\frakT^\theta_t}^\eta) $, $ \beta\curvearrowleft \calD(\calY_{t,\frakT^\theta_t(\omega)}^\theta) $ in the notation of Lemma \ref{composition}) show that for all $ \theta, \eta \in \Theta, t\in [0,T], l\in [0,n]\cap\N_0 $ it holds that 
		\begin{equation}\label{DNNMLPproof2}
			\begin{aligned}
				U_l^\eta\left (\frakT^\theta_t(\omega), Y_{t,\frakT^\theta_t(\omega)}^{\theta,x}(\omega),\omega\right )&=\left(\calR( \Phi_{l,\frakT^\theta_t(\omega)}^\eta)\right) \left( (\calR(\calY_{t,\frakT^\theta_t(\omega)}^\theta))(x)\right)    \\
				&\in \calR\left( \left\{ \Phi \in \mathbf{N}: \calD(\Phi)= \calD(\Phi_{l,\frakT^\theta_t(\omega)}^\eta)\odot\calD(\calY_{t,\frakT^\theta_t(\omega)}^\theta)  \right\}\right)\\
				&=\calR\left( \left\{ \Phi \in \mathbf{N}: \calD(\Phi)= \calD(\Phi_{l,0}^0)\odot\calD(\calY_{0,0}^0)  \right\}\right).\\
			\end{aligned}
		\end{equation}
		Note that Lemma \ref{Identity} (applied for all $ l\in [1,n-1]\cap\N $ with $ H \curvearrowleft (n-l)\left(\dim(\calD(\Phi_f))-2+\mathbb{Y})\right)+1 $ in the notation of Lemma \ref{Identity}) shows for all $ l\in [1,n-1]\cap\N $ that 
		\begin{equation}
			\textup{Id}_\R \in \calR\left( \left\{  \Phi\in \mathbf{N}: \calD(\Phi) = \mathfrak{n}_{(n-l)\left(\dim(\calD(\Phi_f))-2+\mathbb{Y}\right)+1}  \right\}   \right).
		\end{equation}
		This, (\ref{DNNMLPproof2}), and Lemma \ref{composition} (applied for all $ \theta,\eta\in\Theta, t\in[0,T], l\in [0,n-1]\cap\N_0 $ with $ d_1\curvearrowleft d,$  $ d_2 \curvearrowleft1,$  $ d_3 \curvearrowleft1, f\curvearrowleft\textup{Id}_\R $, $ g \curvearrowleft  U_l^\eta\left (\frakT^\theta_t(\omega), Y_{t,\frakT^\theta_t(\omega)}^{\theta,x}(\omega),\omega\right )$, $ \alpha = \mathfrak{n}_{(n-l)\left(\dim(\calD(\Phi_f))-2+\mathbb{Y}\right)+1} $, $ \beta = \calD(\Phi_{l,0}^0)\odot\calD(\calY_{0,0}^0)$ in the notation of Lemma \ref{composition}) show that for all $ \theta,\eta \in \Theta, $ $ t\in[0,T], $ $ l\in[0,n-1]\cap\N_0 $ it holds that 
		\begin{equation}
			U_l^\eta\left (\frakT^\theta_t(\omega), Y_{t,\frakT^\theta_t(\omega)}^{\theta,x}(\omega),\omega\right ) 
			\in \calR\left( \left\{  \Phi\in \mathbf{N}: \calD(\Phi) = \mathfrak{n}_{(n-l)\left(\dim(\calD(\Phi_f))-2+\mathbb{Y}\right)+1}\odot\calD(\Phi_{l,0}^0)\odot\calD(\calY_{0,0}^0) \right\}   \right).
		\end{equation}
		This and Lemma \ref{composition} (applied for all $ \theta,\eta\in\Theta, t\in [0,T], l\in [0,n-1]\cap \N_0 $ with $ d_1 \curvearrowleft d, $  $ d_2 \curvearrowleft 1, $  $ d_3 \curvearrowleft1 $, $ f\curvearrowleft f $, $ g \curvearrowleft U_l^\eta\left (\frakT^\theta_t(\omega), Y_{t,\frakT^\theta_t(\omega)}^{\theta,x}(\omega),\omega\right )  $,  $ \alpha\curvearrowleft \calD(\Phi_f) $, $ \beta\curvearrowleft \mathfrak{n}_{(n-l)\left(\dim(\calD(\Phi_f))-2+\mathbb{Y}\right)+1}\odot\calD(\Phi_{l,0}^0)\odot\calD(\calY_{0,0}^0)$ in the notation of Lemma \ref{composition}) prove that for all $ \eta,\theta \in \Theta, t\in [0,T], l\in [0,n-1]\cap\N_0 $ it holds
		\begin{equation}
			\begin{aligned}
				&f\circ U_l^\eta\left (\frakT^\theta_t(\omega), Y_{t,\frakT^\theta_t(\omega)}^{\theta,x}(\omega),\omega\right )\\
				&\in \calR\left( \left\{  \Phi\in \mathbf{N}: \calD(\Phi) = \calD(\Phi_f)\odot\mathfrak{n}_{(n-l)\left(\dim(\calD(\Phi_f))-2+\mathbb{Y}\right)+1}\odot\calD(\Phi_{l,0}^0)\odot\calD(\calY_{0,0}^0) \right\}   \right).
			\end{aligned}  
		\end{equation}
		Furthermore, (\ref{DNNMLPproof2}) (applied with $ l\curvearrowleft n $) and Lemma \ref{composition} (applied for all $ \theta,\eta\in\Theta, t\in[0,T] $ with $ d_1 \curvearrowleft d, $  $ d_2 \curvearrowleft 1, $  $ d_3 \curvearrowleft1 $, $ f\curvearrowleft f $, $ g \curvearrowleft U_n^\eta\left (\frakT^\theta_t(\omega), Y_{t,\frakT^\theta_t(\omega)}^{\theta,x}(\omega),\omega\right )  $,  $ \alpha\curvearrowleft \calD(\Phi_f) $, $ \beta\curvearrowleft \calD(\Phi_{n,0}^0)\odot\calD(\calY_{0,0}^0)$ in the notation of Lemma \ref{composition}) prove that for all $ \eta,\theta \in \Theta, t\in [0,T] $ it holds
		\begin{equation}
			f\circ U_n^\eta\left (\frakT^\theta_t(\omega), Y_{t,\frakT^\theta_t(\omega)}^{\theta,x}(\omega),\omega\right )
			\in \calR\left( \left\{  \Phi\in \mathbf{N}: \calD(\Phi) = \calD(\Phi_f)\odot\calD(\Phi_{l,0}^0)\odot\calD(\calY_{0,0}^0) \right\}   \right).
		\end{equation}
		To reiterate what was established so far: For each summand of $ U_{n+1}^\theta, \theta\in\Theta, $ existence of a suitable DNN representation was proven. The next step shows that these representing DNNs all have the same depth:
		Note that Lemma \ref{DNNEULER} implies that for all $ t\in[0,T], s\in[t,T] , \theta\in \Theta$ it holds that $ \mathbb{Y}=\dim(\calD(\calY_{t,s}^\theta)). $
		The definition of $ \odot $, the fact that for all $ l\in [0,n]\cap\N_0 $ we have 
		\begin{equation}
			\dim(\calD(\Phi_{l,0}^0)) = l(\dim(\calD(\Phi_f))-2)+(l+1)\mathbb{Y}+\dim(\calD(\Phi_g)) -1
		\end{equation}
		in the induction hypothesis, and the fact that for all $ \theta\in\Theta, t\in[0,T],s\in[t,T] $ we have $ \dim(\calD(\calY_{t,s}^\theta))=\mathbb{Y} $ imply that 
		
		\begin{equation}
			\begin{aligned}
				&\dim(\mathfrak{n}_{(n+1)(\dim(\calD(\Phi_f))-2+\mathbb{Y})+1}\odot\calD(\Phi_g)\odot\calD(\calY_{t,T}^\theta))\\
				&\qquad=\dim(\mathfrak{n}_{(n+1)(\dim(\calD(\Phi_f))-2+\mathbb{Y})+1})+\dim(\calD(\Phi_g))+\dim(\calD(\calY_{t,T}^\theta))-2\\
				&\qquad=(n+1)(\dim(\calD(\Phi_f))-2+\mathbb{Y})+1+\dim(\calD(\Phi_g))+\mathbb{Y}-2\\
				&\qquad=(n+1)(\dim(\calD(\Phi_f))-2)+(n+2)\mathbb{Y}+\dim(\calD(\Phi_g))-1,
			\end{aligned}
		\end{equation}
		that 
		\begin{equation}
			\begin{aligned}
				&\dim(\calD(\Phi_f)\odot\calD(\Phi_{n,0}^0)\odot\calD(\calY_{0,0}^0))\\
				&\qquad= \dim(\calD(\Phi_f))+\dim(\calD(\Phi_{n,0}^0))+\dim(\calD(\calY_{0,0}^0))-2 \\
				&\qquad=\dim(\calD(\Phi_f))+n(\dim(\calD(\Phi_f))-2)+(n+1)\mathbb{Y}+\dim(\calD(\Phi_g)) -1 +\mathbb{Y}-2\\
				&\qquad= (n+1)(\dim(\calD(\Phi_f))-2)+(n+2)\mathbb{Y}+\dim(\calD(\Phi_g))-1,
			\end{aligned}
		\end{equation}
		and for all $ l\in[0,n-1]\cap\N_0 $ that
		\begin{equation}
			\begin{aligned}
				&\dim(\calD(\Phi_f)\odot\mathfrak{n}_{(n-l)\left(\dim(\calD(\Phi_f))-2+\mathbb{Y}\right)+1}\odot\calD(\Phi_{l,0}^0)\odot\calD(\calY_{0,0}^0))\\
				&\qquad = \dim(\calD(\Phi_f))+\dim(\mathfrak{n}_{(n-l)\left(\dim(\calD(\Phi_f))-2+\mathbb{Y}\right)+1})+\dim(\calD(\Phi_{l,0}^0))+\dim(\calD(\calY_{0,0}^0))-3\\
				&\qquad = \dim(\calD(\Phi_f))+(n-l)\left(\dim(\calD(\Phi_f))-2+\mathbb{Y}\right)+1+l(\dim(\calD(\Phi_f))-2)+(l+1)\mathbb{Y}\\
				&\qquad\qquad+\dim(\calD(\Phi_g)) -1+\mathbb{Y}-3\\
				&\qquad= (n+1)(\dim(\calD(\Phi_f))-2)+(n+2)\mathbb{Y}+\dim(\calD(\Phi_g))-1.
			\end{aligned}
		\end{equation}
		Since all DNNs that are necessary to construct $ U_{n+1}^\theta, \theta\in\Theta,  $ have the same depth, Lemma \ref{sumsDNN} and (\ref{DNNMLP_MLP}) imply that there exist $ (\Phi_{n+1,t}^\theta)_{t\in[0,T],\theta\in\Theta}\subseteq\mathbf{N} $ satisfying for all $ \theta\in\Theta, t\in [0,T], x\in\R^d $ it holds that 
		\begin{equation}
			\begin{aligned}\label{DNNMLPproof8}
				&(\mathcal{R}(\Phi_{n+1,t}^\theta))(x)\\
				&\quad=\dfrac{1}{M^{n+1}}\sum_{i=1}^{M^{n+1}}g(Y_{t,T}^{(\theta,0,-i),x}(\omega))\\
				&\qquad+\dfrac{(T-t)}{M}\sum_{i=1}^M\left(f\circ U_{n}^{(\theta,n,i)}\right)\left(\mathfrak{T}_t^{(\theta,n,i)}(\omega),Y_{t,\mathfrak{T}_t^{(\theta,n,i)}(\omega)}^{(\theta,n,i)}(\omega),\omega\right)\\
				&\qquad+\sum_{l=0}^{n-1}\dfrac{(T-t)}{M^{n+1-l}}\sum_{i=1}^{M^{n+1-l}}\left(f \circ  U_l^{(\theta,l,i)}\right)\left(\mathfrak{T}_t^{(\theta,l,i)}(\omega),Y_{t,\mathfrak{T}_t^{(\theta,l,i)}(\omega)}^{(\theta,l,i)}(\omega),\omega\right)\\
				&\qquad-\sum_{l=1}^{n}\dfrac{(T-t)}{M^{n+1-l}}\sum_{i=1}^{M^{n+1-l}}\left(f\circ U_{l-1}^{(\theta,-l,i)}\right)\left(\mathfrak{T}_t^{(\theta,l,i)}(\omega),Y_{t,\mathfrak{T}_t^{(\theta,l,i)}(\omega)}^{(\theta,l,i)}(\omega),\omega\right)\\
				&\quad= U_{n+1}^\theta(t,x,\omega), 
			\end{aligned}
		\end{equation}
		that 
		\begin{equation}\label{DNNMLPproof7}
			\dim(\calD(\Phi_{n+1,t}^\theta))=(n+1)(\dim(\calD(\Phi_f))-2)+(n+2)\mathbb{Y}+\dim(\calD(\Phi_g))-1,
		\end{equation}
		and that 
		\begin{equation}\label{DNNMLPproof3}
			\begin{aligned}
				\calD(\Phi_{n+1,t}^\theta)&=\left( \underset{i=1}{\overset{M^{n+1}}{\boxplus}}\left[\mathfrak{n}_{(n+1)(\dim(\calD(\Phi_f))-2+\mathbb{Y})+1}\odot\calD(\Phi_g)\odot\calD(\calY_{0,0}^0)\right]  \right)\\
				&\quad \boxplus \left( \underset{i=1}{\overset{M}{\boxplus}} \calD(\Phi_f)\odot\calD(\Phi_{n,0}^0)\odot\calD(\calY_{0,0}^0) \right)\\
				&\quad \boxplus \left(\underset{l=0}{\overset{n-1}{\boxplus}} \underset{i=1}{\overset{M^{n+1-l}}{\boxplus}}\left[\calD(\Phi_f)\odot\mathfrak{n}_{(n-l)\left(\dim(\calD(\Phi_f))-2+\mathbb{Y}\right)+1}\odot\calD(\Phi_{l,0}^0)\odot\calD(\calY_{0,0}^0)\right] \right)\\
				&\quad \boxplus \left( \underset{l=1}{\overset{n}{\boxplus}} \underset{i=1}{\overset{M^{n+1-l}}{\boxplus}}\left[\calD(\Phi_f)\odot\mathfrak{n}_{(n-l+1)\left(\dim(\calD(\Phi_f))-2+\mathbb{Y}\right)+1}\odot\calD(\Phi_{l-1,0}^0)\odot\calD(\calY_{0,0}^0)\right] \right).
			\end{aligned}
		\end{equation}
		This shows for all $ t_1,t_2\in[0,T], \theta_1,\theta_2\in\Theta $ that 
		\begin{equation}\label{DNNMLPproof6}
			\calD(\Phi_{n+1,t_1}^{\theta_1})=\calD(\Phi_{n+1,t_2}^{\theta_2}).
		\end{equation}
		Additionally, (\ref{DNNMLPproof3}) and Lemma \ref{triangle} prove for all $ t\in[0,T], \theta\in\Theta $ that 
		\begin{equation}\label{DNNMLPproof4}
			\begin{aligned}
				\mnrm{\calD(\Phi_{n+1,t}^\theta)}&\le\sum_{i=1}^{M^{n+1}} \mnrm{\mathfrak{n}_{(n+1)(\dim(\calD(\Phi_f))-2+\mathbb{Y})+1}\odot\calD(\Phi_g)\odot\calD(\calY_{0,0}^0)}\\
				&\quad + \sum_{i=1}^M \mnrm{\calD(\Phi_f)\odot\calD(\Phi_{n,0}^0)\odot\calD(\calY_{0,0}^0)}\\
				&\quad +\sum_{l=1}^{n-1}\sum_{i=1}^{M^{n+1-l}}\mnrm{\calD(\Phi_f)\odot\mathfrak{n}_{(n-l)\left(\dim(\calD(\Phi_f))-2+\mathbb{Y}\right)+1}\odot\calD(\Phi_{l,0}^0)\odot\calD(\calY_{0,0}^0)}\\
				&\quad +\sum_{l=1}^n\sum_{i=1}^{M^{n+1-l}}\mnrm{\calD(\Phi_f)\odot\mathfrak{n}_{(n-l+1)\left(\dim(\calD(\Phi_f))-2+\mathbb{Y}\right)+1}\odot\calD(\Phi_{l-1,0}^0)\odot\calD(\calY_{0,0}^0)}.
			\end{aligned}
		\end{equation}
		Note that for all $ H_1, H_2, \alpha_0,...,\alpha_{H_1+1}, \beta_0,...,\beta_{H_2+1} \in \N $, $ \alpha, \beta \in \mathbf{D} $ satisfying $ \alpha=(\alpha_0,...,\alpha_{H_1+1}) $ and $ \beta=(\beta_0,...,\beta_{H_2+1}) $ with $ \alpha_0=\beta_{H_2+1} $ it holds that $ \mnrm{\alpha \odot \beta}\le \max\{\mnrm{\alpha},\mnrm{\beta}, 2\alpha_0\}. $
		This, (\ref{DNNMLPproof4}), the fact that for all $ H\in \N $ it holds that $ \mnrm{\mathfrak{n}_{H+2}}=2 $, and the fact that for all $ l\in[0,n]\cap\N_0 $ we have
		\begin{equation}
			\mnrm{\calD(\Phi_{l,0}^0)}\le c(3M)^l
		\end{equation}
		in the induction hypothesis show for all $ t\in[0,T],\theta\in\Theta $ that 
		\begin{equation}\label{DNNMLPproof5}
			\begin{aligned}
				\mnrm{\calD(\Phi_{n+1,t}^\theta)}&\le \left[\sum_{i=1}^{M^{n+1}}c\right]+\left[\sum_{i=1}^M c(3M)^n\right] + \left[\sum_{l=0}^{n-1}\sum_{i=1}^{M^{n+1-l}}c(3M)^l\right] +\left[ \sum_{l=1}^n\sum_{i=1}^{M^{n+1-l}}c(3M)^{l-1} \right]\\
				&=M^{n+1}c +Mc(3M)^n+\left[\sum_{l=0}^{n-1}M^{n+1-l}c(3M)^l\right]+\left[\sum_{l=1}^n M^{n+1-l}c(3M)^{l-1}\right]\\
				&=M^{n+1}c\left[1+3^n+\sum_{l=0}^{n-1}3^l+\sum_{l=1}^n3^{l-1}\right]=M^{n+1}c\left[1+\sum_{l=0}^{n}3^l+\sum_{l=1}^n3^{l-1}\right]\\
				&\le cM^{n+1}\left[1+2\sum_{l=0}^{n}3^l\right]=cM^{n+1}\left[1+2\frac{3^{n+1}-1}{3-1}\right]\\
				&=c(3M)^{n+1}.
			\end{aligned}
		\end{equation}
		(\ref{DNNMLPproof8}),(\ref{DNNMLPproof7}), (\ref{DNNMLPproof6}) and (\ref{DNNMLPproof5}) complete the induction step. By the principle of induction the proof is thus completed.
	\end{proof}
	

	\section{Deep neural network approximations for PDEs}\label{sect4}
	\subsection{Deep neural network approximations with specific polynomial convergence rates}
	The following theorem is our main result. It generalizes  [\citen{Hutzenthaler_2020}, Theorem 4.1] where the special case of semilinear heat equations is treated. In the proof we combine the results from Corollary \ref{fullerror} and Lemma \ref{U=R}. 
	\begin{theorem}\label{MainTheorem}
		Assume Setting \ref{Setting31}, let $T\in (0,\infty), b,c,p\in[1,\infty)$, $ B,\beta\in [0,\infty), $ $ \mathfrak{p}\in \N $, $ \alpha,q \in [2,\infty) $, 
		let $ \nrm{\cdot}_F\colon(\cup_{d\in\N}\R^{d\times d})\rightarrow[0,\infty) $ satisfy for all $ d\in\N, A=(a_{i,j})_{i,j\in\{1,...,d\}}\in\R^{d\times d} $ that $ \nrm{A}_F=\sqrt{\sum_{i,j=1}^d (a_{i,j})^2} $,
		let $ \langle \cdot,\cdot\rangle\colon(\cup_{d \in \N}\R^d\times\R^d)\rightarrow\R $ satisfy for all $ d\in \N, $ $ x=(x_1,...,x_d), $  $y=(y_1,...,y_d)\in\R^d $ that $ \langle x,y\rangle = \sum_{i=1}^d x_iy_i, $
		for all $ d\in \N $ let $ g_d\in C(\R^d,\R), \mu_d = (\mu_{d,i})_{i\in\{1,...,d\}} \in C(\R^d,\R^d), \sigma_d=(\sigma_{d,i,j})_{i,j\in\{1,...,d\}} \in C(\R^d,\R^{d\times d}) $, let $ f\in C(\R,\R) $, assume for all $ v,w\in\R $ that $ |f(v)-f(w)|\le c|v-w|, $\\
		\noindent	for every $ \varepsilon\in(0,1], d\in\N,v\in\R^d $ let $ \hat{\sigma}_{d,\varepsilon}\in C(\R^d,\R^{d\times d}), \frakg_{d,\varepsilon}, \tmu_{d,\varepsilon}, \tsigma_{d,\varepsilon,v}\in\mathbf{N} $, assume for all $ d\in\N, x,y,v\in\R^d,$  $ \varepsilon\in(0,1]      $ that $ \calR(\frakg_{d,\varepsilon})\in C(\R^d,\R), \calR(\tmu_{d,\varepsilon})\in C(\R^d,\R^d) $, $ \calR(\tsigma_{d,\varepsilon,v})\in C(\R^d,\R^{d}) $, 
		 $ \calR(\tsigma_{d,\varepsilon,v})=\hat{\sigma}_{d,\varepsilon}(\cdot)v, $ $ \calD(\tsigma_{d,\varepsilon,v})=\calD(\tsigma_{d,\varepsilon,0}), $ 
		\begin{equation}\label{THM4.2APPROXREGULARITY}
			|(\calR(\frakg_{d,\varepsilon}))(x)|\le b(d^{2c}+\nrm{x}^2)^{\frac{1}{2}}, \quad \max\{\nrm{(\calR(\tmu_{d,\varepsilon}))(0)},\nrm{\hat{\sigma}_{d,\varepsilon}(0)}_F\}\le cd^{c}
		\end{equation}
		\begin{equation}\label{THM4.2APPROXRATE}
			\max\left\{|g_d(x)-(\calR(\frakg_{d,\varepsilon}))(x)|, \nrm{\mu_d(x)-\left(\calR(\tmu_{d,\varepsilon})\right )(x)}, \nrm{\sigma_{d,\varepsilon}(x)-\hat{\sigma}_{d,\varepsilon}(x)}_F\right\}\le \varepsilon B d^p(d^{2c}+\nrm{x}^2)^{q},
		\end{equation}
		\begin{equation}\label{THM4.2APPROX_g_locLIPSCHITZ}
			|(\calR(\frakg_{d,\varepsilon}))(x)-(\calR(\frakg_{d,\varepsilon}))(y)|\le bT^{-\frac{1}{2}}\sqrt{2d^{2c}+\nrm{x}^2+\nrm{y}^2}\nrm{x-y},
		\end{equation}
		\begin{equation}\label{THM4.2APPROXLIPSCHITZ}
			\max\left\{\nrm{(\calR(\tmu_{d,\varepsilon}))(x)-(\calR(\tmu_{d,\varepsilon}))(y)}, \nrm{\hat{\sigma}_{d,\varepsilon}(x)-\hat{\sigma}_{d,\varepsilon}(y)}_F\right\}\le c\nrm{x-y},
		\end{equation}
		\begin{equation}
			\max\left\{\mnrm{\calD(\frakg_{d,\varepsilon})},\mnrm{\calD(\tmu_{d,\varepsilon})}, \mnrm{\calD(\tsigma_{d,\varepsilon,0})}\right\}\le Bd^p\varepsilon^{-\alpha},
		\end{equation}
		\begin{equation}
			\max\left\{\dim(\calD(\frakg_{d,\varepsilon})),\dim(\calD(\tmu_{d,\varepsilon})), \dim(\calD(\tsigma_{d,\varepsilon,0}))\right\} \le Bd^p\varepsilon^{-\beta},
		\end{equation}
		
		\noindent and for all $ d\in \N $ let $ \nu_d $ be a probability measure on $ (\R^d, \mathcal{B}(\R^d)) $ satisfying $ (\int_{\R^d} \nrm{y}^{2q+1}\nu_d(dy))^{\frac{1}{2q+1}}\le Bd^\mathfrak{p}$.		
\noindent		Then
		\begin{enumerate}[(i)]
			\item\label{THM42_item_i} for every $d\in\N$ there exists a unique viscosity solution $ u_d\in \{u\in C([0,T]\times\R^d,\R)\colon$  \\ $ \underset{s\in[0,T]}{\sup}\underset{y\in\R^d}{\sup}\frac{|u(s,y)|}{1+\nrm{y}}<\infty\} $ of 
			\begin{equation}\label{THM4.2_EXVISC}
				(\frac{\partial}{\partial t}u_d)(t,x)+\langle \mu_d(x),(\nabla_xu_d)(t,x)\rangle+\frac{1}{2}\textup{Tr}(\sigma_d(x)[\sigma_d(x)]^*(\textup{Hess}_xu_d)(t,x))+f(u_d(t,x))=0
			\end{equation}
			with $ u_d(T,x)=g_d(x) $ for $ t\in(0,T), x\in\R^d $,
			\item\label{THM42_item_ii} there exist $ C=(C_\gamma)_{\gamma \in (0,1]}\colon(0,1]\rightarrow(0,\infty) ,\eta\in(0,\infty),(\Psi_{d,\varepsilon})_{d\in\N, \varepsilon\in (0,1]}\subseteq \mathbf{N}$ such that for all $ d\in\N, \varepsilon\in(0,1] $ it holds that $ \calR(\Psi_{d,\varepsilon}) \in C(\R^d,\R)$, 
			$ \calP(\Psi_{d,\varepsilon})\le C_\gamma d^{\eta}\varepsilon^{-(6+2\alpha+\beta+\gamma)} ,$ and 
			\begin{equation}
				\left[\int_{\R^d}|u_d(0,x)-(\calR(\Psi_{d,\varepsilon}))(x)|^2\nu_d(dx)\right]^{\frac{1}{2}}\le \varepsilon.
			\end{equation}
		\end{enumerate}
	\end{theorem}
	\begin{proof}
		Throughout the proof assume without loss of generality that 
		\begin{equation}
			B\ge \max \left\{16, |f(0)|+1, 16|c(4c+2|f(0)|)|^{\frac{1}{q-1}}\right\}.
		\end{equation}
		Note that the triangle inequality, (\ref{THM4.2APPROXRATE}) and (\ref{THM4.2APPROXLIPSCHITZ}) imply for all $ d\in\N, x\in\R^d, \varepsilon\in(0,1] $ that 
		\begin{equation}
			\begin{aligned}
				\nrm{\mu_d(x)-\mu_d(y)}&\le \nrm{\mu_d(x)-(\calR(\tmu_{d,\varepsilon}))(x)}+\nrm{(\calR(\tmu_{d,\varepsilon}))(x)-(\calR(\tmu_{d,\varepsilon}))(y)}\\
				&\qquad+\nrm{(\calR(\tmu_{d,\varepsilon}))(y)-\mu_d(y)}\\
				&\le \varepsilon Bd^p\left((d^{2c}+\nrm{x}^2)^{q}+(d^{2c}+\nrm{y}^2)^{q}\right)+c\nrm{x-y}.
			\end{aligned}
		\end{equation}
		This proves for all $ d\in\N,x\in\R^d $  that 
		\begin{equation}\label{THM4.2_2_mu_lipschitz}
			\nrm{\mu_d(x)-\mu_d(y)}\le c\nrm{x-y}.
		\end{equation}
		In the same manner it follows for all $ d\in\N, x\in\R^d $ that
		\begin{equation}\label{THM4.2_2_sigma_lipschitz}
			\nrm{\sigma_d(x)-\sigma_d(y)}_F\le c\nrm{x-y}
		\end{equation} 
		and 
		\begin{equation}
			|g_d(x)-g_d(y)|\le bT^{-\frac{1}{2}}\sqrt{2d^{2c}+\nrm{x}^2+\nrm{y}^2}\nrm{x-y}.
		\end{equation}
		Note that due to the triangle inequality,   (\ref{THM4.2APPROXREGULARITY}), and (\ref{THM4.2APPROXRATE}) it holds for all $ d\in\N, \varepsilon\in(0,1] $ that
		\begin{equation}\label{THM4.2ProofMuSigmaInitialValue}
			\begin{aligned}
				\max\{\nrm{\mu_{d}(0)},\nrm{\sigma_{d}(0)}_F\}&
				\le \max\{\nrm{\mu_{d}(0)-(\calR(\tmu_{d,\varepsilon}))(0)}, \nrm{\sigma_d(0)-\hat{\sigma}_{d,\varepsilon}(0)}_F\}+cd^{c}\\
				&\le \varepsilon B d^{p+2qc}+ cd^{c},
			\end{aligned}
		\end{equation}
		which shows for all $ d\in\N $ that it holds that
		\begin{equation}
			\max\{\nrm{\mu_{d}(0)},\nrm{\sigma_{d}(0)}_F\}\le cd^{c}.
		\end{equation}
		Next, the triangle inequality, (\ref{THM4.2APPROXREGULARITY}),  and (\ref{THM4.2APPROXRATE})  ensure for all $ d\in\N,x\in\R^d, \varepsilon\in(0,1] $ that
		\begin{equation}
			|g_d(x)|\le |g_d(x)-(\calR(\frakg_{d,\varepsilon}))(x)|+|(\calR(\frakg_{d,\varepsilon}))(x)|\le \varepsilon Bd^p(d^{2c}+\nrm{x}^2)^{q}+b(d^{2c}+\nrm{x}^2)^{\frac{1}{2}},
		\end{equation}
		which demonstrates for all $ d\in\N, x\in\R^d $ that 
		\begin{equation}\label{THM4.2_2_proof_g_regularity}
			|g_d(x)|\le b(d^{2c}+\nrm{x}^2)^{\frac{1}{2}}.
		\end{equation}
		Next, [\citen{Hutzenthaler_2020}, Corollary 3.13] (applied for all $ \varepsilon\in (0,1] $ with $ L\curvearrowleft c, q\curvearrowleft q, f\curvearrowleft f, \varepsilon\curvearrowleft\varepsilon $ in the notation of [\citen{Hutzenthaler_2020}, Corollary 3.13]) ensures existence of $ \frakf_\varepsilon \in\mathbf{N}, \varepsilon\in(0,1], $ which satisfy for all $ v,w\in\R, \varepsilon\in(0,1] $ that $ \calR(\frakf_\varepsilon)\in C(\R,\R)$, $ |(\calR(\frakf_\varepsilon))(v)-(\calR(\frakf_\varepsilon)) (w)|\le c|v-w|, $ $ |f(v)-(\calR(\frakf_\varepsilon)) (v)|\le \varepsilon(1+|v|^q), $ $ \dim(\calD(\frakf_\varepsilon))=3 $ and 
		\begin{equation}
			\mnrm{\calD(\frakf_\varepsilon)}\le 16\left(\max\left\{1,|c(4c+2|f(0)|)|^{\frac{1}{(q-1)}}\right\}\right)^{-\frac{q}{q-1}}.
		\end{equation}
		Due to the fact that $ B\ge 1+|f(0)|, $ it holds for all $ \varepsilon\in(0,1] $ that
		\begin{equation}
			|(\calR(\frakf_\varepsilon))(0)|\le |(\calR(\frakf_\varepsilon))(0)-f(0)|+|f(0)|\le \varepsilon+|f(0)|\le B,
		\end{equation}
		which ensures  for all $ \varepsilon\in(0,1] $ that
		\begin{equation}\label{THM4.2_Proof_f_initialvalue}
			\max\{|f(0)|,|(\calR(\frakf_\varepsilon))(0)|\}\le \frac{b}{T}+B =\frac{b+BT}{T}.
		\end{equation}
		Next, for all $ d\in \N $ let $ \varphi_d \in C^2(\R^d,[1,\infty)) $ satisfy for all $ x\in\R^d $ that $ \varphi_d(x)=d^{2c}+\nrm{x}^2. $ Hence for all $ d\in\N, x\in\R^d $ we have 
		\begin{equation}
			(\nabla_x \varphi_d)(x)=2x \quad \text{and}\quad (\textup{Hess}_x \varphi_d)(x)=2\textup{Id}_{\R^d}.
		\end{equation}
		By the Cauchy-Schwartz inequality it thus holds for all $d\in\N, x\in\R^d , z\in\R^d\setminus{\{0\}}$ that
		\begin{equation}
			\frac{|\langle(\nabla_x\varphi_d)(x),z\rangle|}{(\varphi_d(x))^{1/2}\nrm{z}}\le \frac{|2\langle x,z\rangle|}{\nrm{x}\nrm{z}}\le 2 \le 2c,
		\end{equation}
		\begin{equation}
			\frac{\langle z, (\textup{Hess}_x\varphi_d)(x)z\rangle}{(\varphi_d(x))^0 \nrm{z}^2}=\frac{2\nrm{z}^2}{\nrm{z}^2}=2\le 2c.
		\end{equation}
		Next, Jensen's inequality ensures that for all $d\in\N, x\in\R^d $ we have $ (d^{c}+\nrm{x})^2\le 2(d^{2c}+\nrm{x}^2)=2\varphi_d(x) $, which in turn implies that $ d^{c}+\nrm{x}\le \sqrt{2}(\varphi_d(x))^{\frac{1}{2}}. $ This and (\ref{THM4.2ProofMuSigmaInitialValue}) ensure for all $d\in\N, x\in\R^d , \varepsilon\in (0,1]$ that 
		\begin{equation}
			\frac{c\nrm{x}+\max\{\nrm{\mu_d(0)},\nrm{\sigma_d(0)}_F,\nrm{(\calR(\tmu_{d,\varepsilon}))(0)}, \nrm{\hat{\sigma}_{d,\varepsilon}(0)}_F\}}{(\varphi_d(x))^{\frac{1}{2}}}\le \frac{c(d^{c}+\nrm{x})}{(\varphi_d(x))^{\frac{1}{2}}}\le \sqrt{2} c\le 2c.
		\end{equation}
		Next, note that for all $ d\in \N $ it holds that
		\begin{equation}
			 \sup_{r\in(0,\infty)}[\inf_{x\in\R^d,\nrm{x}\ge r} \varphi_d(x)]=\infty  \text{ and }  \inf_{r\in(0,\infty)}\sup_{x\in\R^d, \nrm{x}> r}\frac{|f(0)|+|g_d(x)|}{\varphi_d(x)}=0. 
		\end{equation}
		Lemma \ref{LyapunovLemmaNew} $ (i) $ (applied with for all $ d\in \N $ with $ d \curvearrowleft d, $  $m\curvearrowleft d, $ $ c\curvearrowleft 2c, \kappa \curvearrowleft 1, p\curvearrowleft 2, \varphi \curvearrowleft \varphi_d, \mu \curvearrowleft \mu_d,  $ $ \sigma \curvearrowleft \sigma_d $ in the notation of Lemma \ref{LyapunovLemmaNew}) ensures for all $ d\in \N, x\in \R^d $ that
		\begin{equation}
			\langle \mu(x), (\nabla \varphi_d)(x)\rangle +\frac{1}{2}\textup{Tr}(\sigma(x)[\sigma(x)]^*(\textup{Hess}\varphi_d)(x))\le 12c^3.
		\end{equation}
		Analogously we get for all $ d\in \N, x\in \R^d, \varepsilon\in (0,1] $ that
		\begin{equation}
			\langle (\calR(\tmu_{d,\varepsilon}))(x), (\nabla \varphi_d)(x)\rangle +\frac{1}{2}\textup{Tr}(\hat{\sigma}_{d,\varepsilon}(x)[\hat{\sigma}_{d,\varepsilon}(x)]^*(\textup{Hess}\varphi_d)(x))\le 12c^3.
		\end{equation}
	Let $ (\Omega, \mathcal{F}, \mathbb{P}, (\F_t)_{t\in [0,T]}) $ be a filtered probability space satisfying the usual conditions, for all $ d\in \N  $ let $ W^d\colon[0,T]\times\Omega\rightarrow\R^d $ be a standard $ (\F_t)_{t\in[0,T]} $-Brownian motion. Then [\citen{Beck_2021}, Corollary 3.2](applied twice: the first application for all $ d\in \N $ with $ d\curvearrowleft d $, $ m\curvearrowleft d, $ ,$ T\curvearrowleft T $, $ L\curvearrowleft c $, $ \rho \curvearrowleft 12c^3 $, $ \mathcal{O}\curvearrowleft \R^d $, $ \mu \curvearrowleft ([0,T]\times \R^d \ni (t,x)\rightarrow \mu_d(x)\in \R^d), $ $ \sigma \curvearrowleft ([0,T]\times \R^d \ni (t,x)\rightarrow \sigma_d(x)\in \R^{d\times d}) $, $ g\curvearrowleft g_d, $  $ f\curvearrowleft ([0,T]\times \R^d\times \R \ni (t,x,v)\rightarrow f(v)\in \R), $ $ V\curvearrowleft \varphi_d $ in the notation of [\citen{Beck_2021}, Corollary 3.2]; the second application for all $ d\in \N, \varepsilon\in (0,1] $ with $ d\curvearrowleft d $, $ m\curvearrowleft d, $ ,$ T\curvearrowleft T $, $ L\curvearrowleft c $, $ \rho \curvearrowleft 12c^3 $, $ \mathcal{O}\curvearrowleft \R^d $, $ \mu \curvearrowleft ([0,T]\times \R^d \ni (t,x)\rightarrow(\calR(\tmu_{d,\varepsilon}))(x)\in \R^d), $ $ \sigma \curvearrowleft ([0,T]\times \R^d \ni (t,x)\rightarrow\hat{\sigma}_{d,\varepsilon}(x)\in \R^{d\times d}) $, $ g \curvearrowleft \calR(\frakg_{d,\varepsilon}), $ $ f\curvearrowleft \calR(\frakf_\varepsilon) $, $  V\curvearrowleft \varphi_d  $ in the notation of [\citen{Beck_2021}, Corollary 3.2]) guarantees that 
	\begin{enumerate}[(I)]
		\item\label{THM4.2_ENUM_Inew} for every $d\in\N$ there exists a unique viscosity solution\\ $u_d\in\{ u\in C([0,T]\times \R^d,\R):\limsup_{r\rightarrow\infty} [ \sup_{s\in[0,T]}\sup_{x\in\R^d,\nrm{x}>r}\frac{|u(s,y)|}{\varphi_d(x)}]=0\} $ of 
		\begin{equation}
			(\frac{\partial}{\partial t}u_d)(t,x)+\langle \mu_d(x),(\nabla_xu_d)(t,x)\rangle+\frac{1}{2}\textup{Tr}(\sigma_d(x)[\sigma_d(x)]^*(\textup{Hess}_xu_d)(t,x))=-f(u_d(t,x))
		\end{equation}
		with $ u_d(T,x)=g_d(x) $ for $ t\in(0,T), x\in\R^d $,
		\item\label{THM4.2_ENUM_IInew} for every $ d\in\N, x\in\R^d, t\in[0,T],\varepsilon\in(0,1] $ there exist  up to indistinguishability unique $ (\F_s)_{s\in[t,T]} $-adapted stochastic processes $ X_{t}^{x,d}=(X_{t,s}^{x,d})_{s\in[t,T]}\colon[t,T]\times \Omega\rightarrow\R^d $ and $ \calX_t^{x,d,\varepsilon}=(\calX_{t,s}^{x,d,\varepsilon})_{s\in[t,T]}\colon[t,T]\times\Omega\rightarrow\R^d $ with continuous sample paths satisfying that for all $ s\in[t,T] $ we have $ \mathbb{P}$-a.s. that 
		\begin{equation}
			\begin{aligned}
				X_{t,s}^{x,d}&=x+\int_t^s\mu_d(X_{t,r}^{x,d})dr+\int_t^s\sigma_d(X_{t,r}^{x,d})dW^d_r,\\
				\calX_{t,s}^{x,d,\varepsilon}&=x+\int_t^s(\calR(\tmu_{d,\varepsilon}))(\calX_{t,r}^{x,d,\varepsilon})dr+\int_t^s\hat{\sigma}_{d,\varepsilon}(\calX_{t,r}^{x,d,\varepsilon})dW^d_r,
			\end{aligned}
		\end{equation}
		\item\label{THM4.2_ENUM_IIInew} for every $ d\in\N, \varepsilon\in(0,1] $ there exist unique $ v_d\in C([0,T]\times \R^d,\R) $ and $ \frakv_{d,\varepsilon}\in C([0,T]\times\R^d,\R) $ which satisfy for all $ t\in[0,T],x\in\R^d $ that 
		\begin{equation}
			\begin{aligned}
				&\left[\sup_{s\in[0,T]}\sup_{y\in\R^d}\left(\frac{v_d(s,y)}{(\varphi_d(y))^{\frac{1}{2}}}\right)\right]+\left[\sup_{s\in[0,T]}\sup_{y\in\R^d}\left(\frac{\frakv_{d,\varepsilon}(s,y)}{(\varphi_d(y))^{\frac{1}{2}}}\right)\right]+\E\left[\left|(\calR(\frakg_{d,\varepsilon}))(\calX_{t,T}^{x,d,\varepsilon})\right|\right]\\
				&+\E\left[\left|g_d(X_{t,T}^{x,d})\right|\right]+\int_t^T\E\left[\left|f(v_d(s,X_{t,s}^{x,d}))\right|\right]ds+\int_t^T\E\left[\left|(\calR(\frakf_\varepsilon))(\frakv_{d,\varepsilon}(s,\calX_{t,s}^{x,d,\varepsilon}))\right|\right]ds<\infty
			\end{aligned}
		\end{equation}
		\begin{equation}
			v_d(t,x)=\E\left[g_d(X_{t,T}^{x,d})\right]+\int_t^T\E\left[f(v_d(s,X_{t,s}^{x,d}))\right]ds,
		\end{equation}
		\begin{equation}
			\frakv_{d,\varepsilon}(t,x)=\E\left[(\calR(\frakg_{d,\varepsilon}))(\calX_{t,T}^{x,d,\varepsilon})\right]+\int_t^T\E\left[(\calR(\frakf_\varepsilon))(\frakv_{d,\varepsilon}(s,\calX_{t,s}^{x,d,\varepsilon}))\right]ds,
		\end{equation}
		\item\label{THM4.2_ENUM_IVnew} we have for all $ d\in\N, x\in\R^d, t\in[0,T]$ that $ u_d(t,x)=v_d(t,x). $
	\end{enumerate}
Item (\ref{THM4.2_ENUM_IInew}) and Lipschitz continuity of $ \mu_d, \sigma_d, \calR(\tmu_{d,\varepsilon}), \hat{\sigma}_{d,\varepsilon}, d\in \N, \varepsilon\in (0,1], $ and , e.g. [\citen{rogers_williams_vol2}, Lemma 13.6], ensure for all $ d\in \N, \varepsilon\in (0,1],  $ $t\in[0,T], s\in[t,T],r\in[s,T] $ that 
\begin{equation}\label{THM4.2_PATHWISE_UNIQUE}
	\mathbb{P}(X_{t,r}^{x,d}=X^{X^{x,d}_{t,s},d}_{s,r})=\mathbb{P}(\calX_{t,r}^{x,d,\varepsilon}=\calX^{\calX^{x,d,\varepsilon}_{t,s},d,\varepsilon}_{s,r})=1.
\end{equation}
Next, item (\ref{THM4.2_ENUM_IIInew}) ensures for all $ d\in \N, x\in \R^d, t\in [0,T], s\in [t,T] $ that $ u_d(s,X_{t,s}^{x,d}) $ is integrable.
This, (\ref{THM4.2_PATHWISE_UNIQUE}), (\ref{THM4.2_2_proof_g_regularity}), (\ref{THM4.2_Proof_f_initialvalue}), and Lemma \ref{LyapunovLemmaNew} $ (ii) $ (applied for all $ d\in \N $ with $ d\curvearrowleft d $, $ m\curvearrowleft d, c\curvearrowleft 2c, $  $ \kappa \curvearrowleft 1, p \curvearrowleft 2,$  $ \varphi \curvearrowleft \varphi_d, \mu \curvearrowleft \mu_d, \sigma\curvearrowleft \sigma_d $ in the notation of Lemma \ref{LyapunovLemmaNew}) prove for all $ d\in\N, x\in \R^d, t\in [0,T], s\in [t,T] $ that 
\begin{equation}
	\begin{aligned}
		\E\left[\left|u_d(s,X_{t,s}^{x,d})\right|\right]&\le\E\left[\left|g_d\left (X^{X_{t,s}^{x,d},d}_{s,T}\right ) \right|\right]+\int_s^T\E\left[\left|f\left(u_d\left(r,X_{s,r}^{X_{t,s}^{x,d},d}\right)\right)\right|\right]dr\\
		&=\E\left[\left|g_d(X_{t,T}^{x,d})\right|\right]+\int_s^T\E\left[\left| f(u_d(r,X_{t,r}^{x,d})) \right|\right]dr\\
		&\le b\E\left[ \left(d^{2c}+\nrm{X_{t,T}^{x,d}}^2 \right)^{\frac{1}{2}}\right]+(T-s)|f(0)|+c\int_s^T\E\left[|u_d(r,X_{t,r}^{x,d})|\right]dr\\
		&\le be^{6c^3(T-t)}(d^{2c}+\nrm{x}^2)^{\frac{1}{2}}+b+BT+c\int_s^T\E\left[|u_d(r,X_{t,r}^{x,d})|\right]dr.
	\end{aligned}
\end{equation}
This and Gronwalls inequality ensure for all $ d\in \N, x\in \R^d, t\in [0,T] , s\in [t,T]$ that 
\begin{equation}
	\E\left[\left|u_d(s,X_{t,s}^{x,d})\right|\right]\le 2(b+BT)e^{7c^3(T-t)}(d^{2c}+\nrm{x}^2)^{\frac{1}{2}}.
\end{equation}
This ensures for all $ d\in \N $ that 
\begin{equation}\label{THM4.2_u_linear_growth}
	\sup_{s\in[0,T]}\sup_{y\in \R^d} \frac{|u_d(s,y)|}{1+\nrm{y}}<\infty,
\end{equation}
which combined with item (\ref{THM4.2_ENUM_Inew}) establish item (\ref{THM42_item_i}).
 It follows analogously for all $ d\in \N , \varepsilon\in (0,1]$ that
 \begin{equation}
 	\sup_{s\in[0,T]}\sup_{y\in \R^d}\frac{|\frakv_{d,\varepsilon}(s,y)|}{1+\nrm{y}}<\infty.
 \end{equation}
In order to prove item $ (\ref{THM42_item_ii}) $, we will combine Lemma \ref{U=R} with Corollary \ref{fullerror}, for which we yet need a suitable MLP setting.\\

\noindent Let $ \Theta=\cup_{n\in \N} \mathbb{Z}^n$, let $ \mathfrak{t}^\theta\colon\Omega\rightarrow[0,1] , \theta\in\Theta$, be i.i.d. random variables,
assume for all $ t\in(0,1)$ that $ \mathbb{P}(\mathfrak{t}^0\le t)=t $, let $ \mathfrak{T}^\theta\colon[0,T]\times\Omega\rightarrow[0,T] $ satisfy for all $ \theta\in\Theta, t\in[0,T] $ that $ \mathfrak{T}_t^\theta =t+(T-t)\mathfrak{t}^\theta $,
let $ W^{\theta,d}\colon[0,T]\times\Omega\rightarrow\R^d, \theta \in \Theta, d\in \N$, be i.i.d. standard $ (\F_t)_{t\in[0,T]}$-Brownian motions, assume that $(\mathfrak{t}^\theta)_{\theta\in\Theta} $ and $ (W^{\theta,d})_{\theta \in \Theta,d\in\N} $ are independent,
for every $ \delta \in (0,1] $, $d\in\N,  N\in \N, \theta \in \Theta, x\in \R^d, t\in [0,T] $ let $ Y_t^{N,d,\theta,x,\delta}=(Y_{t,s}^{N,d,\theta,x,\delta})_{s\in[t,T]}\colon[t,T]\times\Omega\rightarrow\R^d $ satisfy for all $d\in\N, n\in \{0,1,...,N\}, s\in [\tfrac{nT}{N},\tfrac{(n+1)T}{N}]\cap[t,T] $ that $ Y_{t,t}^{N,d,\theta,x,\delta}=x $ and 
\begin{equation}
	\begin{aligned}
		&Y_{t,s}^{N,d,\theta,x,\delta}-Y_{t,\max\{t,\frac{nT}{N}\}}^{N,d,\theta,x,\delta}\\
		& =	(\calR(\tmu_{d,\delta}))(Y_{t,\max\{t,\frac{nT}{N}\}}^{N,d,\theta,x,\delta})(s-\max\{t,\frac{nT}{N}\})+\hat{\sigma}_{d,\delta}(Y_{t,\max\{t,\frac{nT}{N}\}}^{N,d,\theta,x,\delta})(W_s^{\theta,d}-W_{\max\{t,\frac{nT}{N}\}}^{\theta,d}),
	\end{aligned}
\end{equation}
let $ U_{n,M,d,\delta}^\theta\colon[0,T]\times\R^d\times\Omega\rightarrow\R, n,M\in \mathbb{Z},d\in\N, \theta\in\Theta,\delta\in(0,1]$ satisfy for all $d\in\N, \theta \in \Theta, n\in\N_0, t\in [0,T], x\in\R^d,\delta\in(0,1] $ that
\begin{equation}
	\begin{aligned}
		&U_{n,M,d,\delta}^\theta(t,x)=\dfrac{\mathbbm{1}_\N(n)}{M^n}\sum_{i=1}^{M^n}(\calR(\frakg_{d,\delta}))(Y_{t,T}^{M^M,d,(\theta,0,-i),x,\delta})\\
		&+\sum_{l=0}^{n-1}\dfrac{(T-t)}{M^{n-l}}\left[\sum_{i=1}^{M^{n-l}}(\calR(\frakf_{\delta})\circ U_{l,M,d}^{(\theta,l,i)}-\mathbbm{1}_\N(l)\calR(\frakf_\delta)\circ U_{l-1,M,d}^{(\theta,-l,i)})(\mathfrak{T}_t^{(\theta,l,i)},Y_{t,\mathfrak{T}_t^{(\theta,l,i)}}^{M^M,d,(\theta,l,i),x,\delta})\right].
	\end{aligned}
\end{equation}
\noindent Let $ k_{d,\varepsilon}\in\N, d\in\N, \varepsilon\in(0,1] $ satisfy for all $ d\in\N, \varepsilon\in (0,1], $ that
\begin{equation}
	k_{d,\varepsilon}=\max\{\mnrm{\calD(\frakf_\varepsilon)},\mnrm{\calD(\frakg_{d,\varepsilon})},\mnrm{\calD(\tmu_{d,\varepsilon})},\mnrm{\calD(\tsigma_{d,\varepsilon,0})}, 2\},
\end{equation}
let $ c_d\in[1,\infty), d\in\N, $ satisfy for all $ d\in \N $
\begin{equation}
	\begin{aligned}
	c_d=8^qb^qc^2d^{(c+\mathfrak{p})(q+1)}B^q(T+1)e^{q(q32c^4+24c^3)T+(4ct)^q+(2c+1)^2},
	\end{aligned}
\end{equation}
let $ N_{d,\varepsilon}\in\N, d\in\N, \varepsilon\in(0,1] $ satisfy for all $ d\in\N, \varepsilon\in (0,1] $
\begin{equation}
	N_{d,\varepsilon}=\min\left\{n\in\N\cap[2,\infty): c_d\left[\frac{\exp(4ncT+\tfrac{n}{2})}{n^{n/2}}+\frac{1}{n^{n/2}}\right]\le \frac{\varepsilon}{2}\right\},
\end{equation} 
let $ \mathfrak{C}=(\mathfrak{C}_\gamma)_{\gamma\in(0,1]}\colon(0,1]\rightarrow(0,\infty] $ satisfy for all $ \gamma\in(0,1] $ that
\begin{equation}
	\mathfrak{C}_\gamma =\sup_{n\in\N\cap[2,\infty)}\left[(3n)^{3n+1}\left( \frac{e^{[4cT+\tfrac{1}{2}](n-1)}+1}{(n-1)^{(n-1)/2}}\right)^{6+\gamma}\right],
\end{equation}
and for all $ d\in\N, \varepsilon\in (0,1] $ let $ \delta_{d,\varepsilon}=\frac{\varepsilon}{4Bd^pc_d}. $
Observe that for all $ \gamma\in(0,1] $ it holds that
\begin{equation}\label{proof41estimatefrakC}
	\begin{aligned}
		\mathfrak{C}_\gamma&\le \sup_{n\in\N\cap[2,\infty)}\left[(3n)^{3n+1}\left( \frac{e^{[4cT+\tfrac{1}{2}](n-1)+1}}{(n-1)^{(n-1)/2}}\right)^{6+\gamma}\right]\\
		&\le \sup_{n\in\N\cap[2,\infty)}\left[(3n)^{3n+1}\left( \frac{e^{4cT+\tfrac{3}{2}}}{(n-1)^{1/2}}\right)^{(n-1)(6+\gamma)}\right]\\
		&= \sup_{n\in\N\cap[2,\infty)}\left[n^4 3^{3n+1} \frac{e^{[4cT+\tfrac{3}{2}](n-1)(6+\gamma)}}{(n-1)^{\tfrac{\gamma}{2}(n-1)}}\cdot\left(\frac{n}{n-1}\right)^{3(n-1)}\right]\\
		&\le \sup_{n\in\N\cap[2,\infty)}\left[n^4 3^{3n+1} \left( \frac{e^{[4cT+\tfrac{3}{2}](6+\gamma)}}{(n-1)^{\tfrac{\gamma}{2}}}\right)^{(n-1)}\right]\sup_{n\in\N\cap[2,\infty)}\left[\left(\frac{n}{n-1}\right)^{3(n-1)}\right]\\
		&<\infty.
	\end{aligned}
\end{equation}
Note that the definition of $ (c_d)_{d\in\N} $ implies that there exists a $ \tilde{c}\in[1,\infty) $ such that for all $ d\in\N $ it holds that 
\begin{equation}\label{THM42_proof_c}
	\begin{aligned}
		c_d&=\tilde{c}d^{(c+\mathfrak{p})(q+1)}.
	\end{aligned}
\end{equation}
Corollary \ref{fullerror} (applied for every $ d,M\in\N, n\in\N_0, x\in\R^d, \delta\in(0,1] $ with $ \delta \curvearrowleft \delta Bd^p, c\curvearrowleft 2c, b\curvearrowleft b+BT, p\curvearrowleft 2, q\curvearrowleft q,$ $ \varphi\curvearrowleft \varphi_d, $ $ f_1\curvearrowleft([0,T]\times\R^d\times\R\ni (t,x,v)\rightarrow f(v)\in\R), f_2 \curvearrowleft ([0,T]\times\R^d\times\R\ni (t,x,v)\rightarrow (\calR (\frakf_\delta))(v)\in\R)$), $ g_1\curvearrowleft g_d,$  $ g_2\curvearrowleft \calR(\frakg_{d,\delta}),$  $\mu_1\curvearrowleft \mu_d $, $ \mu_2\curvearrowleft\calR(\tmu_{d,\delta}), $  $ \sigma_1\curvearrowleft \sigma_d $, $ \sigma_2\curvearrowleft\hat{\sigma}_{d,\delta}, u_1\curvearrowleft u_d $ in the notation of Corollary \ref{fullerror}) shows that for all $ d,M\in\N,n\in\N_0, x\in\R^d, \delta\in(0,1] $ it holds that 
\begin{equation}
	\begin{aligned}
		\left(\E\left[|U_{n,M,d,\delta}^0(0,x)-u_d(0,x)|^2\right]\right)^{\frac{1}{2}}
		&
		\le 4^{q+1}b^qc^2(T+1)e^{q(q32c^4+24c^3)T+(4cT)^q+(2c+1)^2}\\
		&\cdot \left(d^{2c}+\nrm{x}^2\right)^{\frac{q}{2}+\frac{1}{4}}\left[\delta Bd^p+\frac{e^{4ncT+\frac{M}{2}}}{M^{\frac{n}{2}}}+\frac{1}{M^{\frac{M}{2}}}\right].
	\end{aligned}
\end{equation}
This, the triangle inequality, Jensen's inequality, and the fact that for all $ d\in \N $ it holds that $ \nu_d $ is a probability measure on $ (\R^d, \mathcal{B}(\R^d)) $ satisfying $ (\int_{\R^d}||y||^{2q}\nu_d(dy) )^{1/(2q)}\le Bd^\mathfrak{p}$  imply 
\begin{equation}
	\begin{aligned}
		&\left[\int_{\R^d}\E\left[|U_{n,M,d,\delta}^0(0,x)-u_d(0,x)|^2\right]\nu_d(dx)\right]^{\frac{1}{2}}\\
		&\le 4^{q+1}b^qc^2(T+1)e^{q(q32c^4+24c^3)T+(4cT)^q+(2c+1)^2}\\
		&\cdot\left(\delta Bd^p+\frac{e^{4ncT+\frac{M}{2}}}{M^{\frac{n}{2}}}+\frac{1}{M^{\frac{M}{2}}}\right)\left(\int_{\R^d}\left(d^{2c}+\nrm{x}^2\right)^{q+\frac{1}{2}}\nu_d(dx)\right)^{\frac{1}{2}}\\
		&\le 4^{q+1}b^qc^2(T+1)e^{q(q32c^4+24c^3)T+(4cT)^q+(2c+1)^2}\\
		&\cdot\left(\delta Bd^p+\frac{e^{4ncT+\frac{M}{2}}}{M^{\frac{n}{2}}}+\frac{1}{M^{\frac{M}{2}}}\right)\left(2^{q-\frac{1}{2}}d^{c(2q+1)}+2^{q-\frac{1}{2}}\int_{\R^d}\nrm{x}^{2q+1}\nu_d(dx)\right)^{\frac{1}{2}}\\
		&\le c_d\left(\delta Bd^p+\frac{e^{4ncT+\frac{M}{2}}}{M^{\frac{n}{2}}}+\frac{1}{M^{\frac{M}{2}}}\right).
	\end{aligned}
\end{equation}
This, the definitions of $ (N_{d,\varepsilon})_{d\in\N,\varepsilon\in (0,1]}, (\delta_{d,\varepsilon})_{d\in\N, \varepsilon\in (0,1]}, $ and Fubini's theorem show for all $d\in\N, \varepsilon\in(0,1] $ that 
\begin{equation}
	\begin{aligned}
		&\E\left[\int_{\R^d}|U^0_{N_{d,\varepsilon},N_{d,\varepsilon}, d,\delta_{d,\varepsilon}}(0,x)-u_d(0,x)|^2\nu_d(dx)\right]\\
		&=\int_{\R^d}\E\left[|U^0_{N_{d,\varepsilon},N_{d,\varepsilon}, d,\delta_{d,\varepsilon}}(0,x)-u_d(0,x)|^2\right]\nu_d(dx) \le \left(\frac{\varepsilon}{4}+\frac{\varepsilon}{2}\right)^2\le \varepsilon^2.
	\end{aligned}
\end{equation}
This shows that for all $ d\in\N, \varepsilon\in (0,1] $ there exists a $ \omega_{d,\varepsilon}\in\Omega $ such that
\begin{equation}
	\int_{\R^d}|U^0_{N_{d,\varepsilon},N_{d,\varepsilon}, d,\delta_{d,\varepsilon}}(0,x,\omega_{d,\varepsilon})-u_d(0,x)|^2\nu_d(dx)\le \varepsilon^2.
\end{equation}
This and Lemma \ref{U=R} (applied for all $ d\in\N,v\in\R^d,\varepsilon\in (0,1] ,$ with $ c\curvearrowleft k_{d,\delta_{d,\varepsilon}} $, $ d\curvearrowleft d $, $ K\curvearrowleft N_{d,\varepsilon}^{N_{d,\varepsilon}} $, $ M\curvearrowleft N_{d,\varepsilon},$  $ T\curvearrowleft T  $, $ \Phi_\mu \curvearrowleft \tmu_{d,\delta_{d,\varepsilon}} $, $ \Phi_{\sigma,v}\curvearrowleft \tsigma_{d,\delta_{d,\varepsilon},v} $, $ \Phi_f \curvearrowleft \frakf_{\delta_{d,\varepsilon}} $, $ \Phi_g \curvearrowleft \frakg_{d,\delta_{d,\varepsilon}} $ in the notation of Lemma \ref{U=R}) imply that for all $ d\in \N, \varepsilon\in (0,1] $ there exists a deep neural network $ \Psi_{d,\varepsilon} \in \mathbf{N} $ such that for all $ x\in \R^d $ it holds that $ \calR(\Psi_{d,\varepsilon})\in C(\R^d,\R), $ $ (\calR(\Psi_{d,\varepsilon}))(x)=U^0_{N_{d,\varepsilon},N_{d,\varepsilon}, d,\delta_{d,\varepsilon}}(0,x,\omega_{d,\varepsilon}), $ $ \mnrm{\calD(\Psi_{d,\varepsilon})}\le k_{d,\delta_{d,\varepsilon}}(3N_{d,\varepsilon})^{N_{d,\varepsilon}} $, and 
\begin{equation}
	\begin{aligned}
		\dim(\calD(\Psi_{d,\varepsilon}))&\le (N_{d,\varepsilon}+1)N_{d,\varepsilon}^{N_{d,\varepsilon}}\left(\max\{\dim(\calD(\tmu_{d,\delta_{d,\varepsilon}})),\dim(\calD(\tsigma_{d,\varepsilon,0}))\}-1\right)\\
		&\quad+N_{d,\varepsilon}+\dim(\calD(\frakg_{d,\delta_{d,\varepsilon}}))-1\\
		&\le N_{d,\varepsilon}+Bd^p\delta_{d,\varepsilon}^{-\beta}+(N_{d,\varepsilon}+1)N_{d,\varepsilon}^{N_{d,\varepsilon}}(Bd^p\delta_{d,\varepsilon}^{-\beta}-1)-1\\
		&\le N_{d,\varepsilon}Bd^p\delta_{d,\varepsilon}^{-\beta}+(N_{d,\varepsilon}+1)^{N_{d,\varepsilon}+1}(Bd^p\delta^{-\beta}_{d,\varepsilon})\\
		&\le 2(N_{d,\varepsilon}+1)^{N_{d,\varepsilon}+1}Bd^p\delta_{d,\varepsilon}^{-\beta}.
	\end{aligned}
\end{equation}
Hence we get for all $ d\in \N, \varepsilon\in (0,1] $ that
\begin{equation}
	\begin{aligned}
		\calP(\Psi_{d,\varepsilon})&\le \sum_{j=1}^{\dim(\calD(\Psi_{d,\varepsilon}))}k_{d,\delta_{d,\varepsilon}}(3N_{d,\varepsilon})^{N_{d,\varepsilon}}(k_{d,\delta_{d,\varepsilon}}(3N_{d,\varepsilon})^{N_{d,\varepsilon}}+1)\\
		&=\dim(\calD(\Psi_{d,\varepsilon}))\left[k_{d,\delta_{d,\varepsilon}}^2(3N_{d,\varepsilon})^{2N_{d,\varepsilon}}+k_{d,\delta_{d,\varepsilon}}(3N_{d,\varepsilon})^{N_{d,\varepsilon}}\right]\\
		&\le 2(N_{d,\varepsilon}+1)^{N_{d,\varepsilon}+1}Bd^p\delta_{d,\varepsilon}^{-\beta}\left[2k_{d,\delta_{d,\varepsilon}}^2(3N_{d,\varepsilon})^{2N_{d,\varepsilon}}\right]\\
		&\le 4Bd^p\delta_{d,\varepsilon}^{-\beta}k_{d,\delta_{d,\varepsilon}}^2(3N_{d,\varepsilon})^{3N_{d,\varepsilon}+1}.
	\end{aligned}
\end{equation}
Note that the definition of $ (k_{d,\varepsilon})_{d\in\N,\varepsilon\in(0,1]} $ implies for all $ d\in\N, \varepsilon\in(0,1] $ that
\begin{equation}
	\begin{aligned}
		k_{d,\varepsilon}&= \max\{\mnrm{\calD(\frakf_\varepsilon)},\mnrm{\calD(\frakg_{d,\varepsilon})},\mnrm{\calD(\tmu_{d,\varepsilon})},\mnrm{\calD(\tsigma_{d,\varepsilon,0})},2\}\\
		&\le \max\{B\varepsilon^{-2}, Bd^p\varepsilon^{-\alpha},2\}\\
		&\le \max\{B\varepsilon^{-\alpha}, Bd^p\varepsilon^{-\alpha},2\}\quad\le Bd^p\varepsilon^{-\alpha}.
	\end{aligned}
\end{equation}
This shows that for all $ d\in\N, \varepsilon\in(0,1] $ it holds that
\begin{equation}
	\begin{aligned}
		\calP(\Psi_{d,\varepsilon})&\le 4Bd^p\delta_{d,\varepsilon}^{-\beta}k_{d,\delta_{d,\varepsilon}}^2(3N_{d,\varepsilon})^{3N_{d,\varepsilon}+1}\\
		&\le 4B^3d^{3p}\delta_{d,\varepsilon}^{-2\alpha-\beta}(3N_{d,\varepsilon})^{3N_{d,\varepsilon}+1}.\\
	\end{aligned}
\end{equation}
Next, note that the definition of $ (N_{d,\varepsilon})_{d\in\N, \varepsilon\in (0,1]} $ implies that for all $ d\in\N,\varepsilon\in (0,1] $ it holds that
\begin{equation}
	\varepsilon\le 2c_d\frac{\exp((4cT+\tfrac{1}{2})(N_{d,\varepsilon}-1))+1}{(N_{d,\varepsilon}-1)^{(N_{d,\varepsilon}-1)/2}}.
\end{equation}
This  and the definition of $ (\mathfrak{C}_\gamma)_{\gamma\in(0,1]} $ imply for all $ d\in\N, \gamma,\varepsilon\in(0,1] $ that
\begin{equation}
	\begin{aligned}
		\calP(\Psi_{d,\varepsilon})&\le4B^3d^{3p}\delta_{d,\varepsilon}^{-2\alpha-\beta}(3N_{d,\varepsilon})^{3N_{d,\varepsilon}+1}\varepsilon^{6+\gamma}\varepsilon^{-6-\gamma}\\
		&\le 4B^3d^{3p}\delta_{d,\varepsilon}^{-2\alpha-\beta}\varepsilon^{-6-\gamma}(3N_{d,\varepsilon})^{3N_{d,\varepsilon}+1}\left(2c_d\frac{\exp((4cT+\tfrac{1}{2})(N_{d,\varepsilon}-1))+1}{(N_{d,\varepsilon}-1)^{(N_{d,\varepsilon}-1)/2}} \right)^{6+\gamma}\\
		&\le 
		2^{8+\gamma}B^3d^{3p}\delta_{d,\varepsilon}^{-2\alpha-\beta}\varepsilon^{-6-\gamma}c_d^{6+\gamma}\left[\sup_{n\in\N\cap[2,\infty)}(3n)^{3n+1}\left(\frac{e^{(4cT+\frac{1}{2})(n-1)}+1}{(n-1)^{\frac{n-1}{2}}}\right)^{6+\gamma}\right]\\
		&=2^{8+\gamma}B^3d^{3p}\delta_{d,\varepsilon}^{-2\alpha-\beta}\varepsilon^{-6-\gamma}c_d^{6+\gamma}\mathfrak{C}_\gamma\\
		&=2^{8+4\alpha+2\beta+\gamma}B^{3+2\alpha+\beta}\mathfrak{C}_\gamma c_d^{6+2\alpha+\beta+\gamma}d^{p(3+2\alpha+\beta)}\varepsilon^{-(6+2\alpha+\beta+\gamma)}.
	\end{aligned}
\end{equation}
Combining this with (\ref{proof41estimatefrakC}) and (\ref{THM42_proof_c}) ensures  that there exist  $ C=(C_\gamma)_{\gamma\in(0,1]}\colon(0,1]\rightarrow[0,\infty) $ and $ \eta\in(0,\infty) $ such that for all $ d\in\N, \gamma,\varepsilon\in (0,1] $ it holds that $ \calP(\Psi_{d,\varepsilon})\le C_\gamma d^\eta \varepsilon^{-(2\alpha+\beta+\gamma+6)} $. This establishes item $ (\ref{THM42_item_ii})$ and thus completes the proof.
\end{proof}

	\subsection{Deep neural network approximations with positive polynomial convergence rates}
	This corollary is an application of Theorem \ref{MainTheorem}, where for all $ d\in\N $ we set the probability measure $ \nu_d $ to be the $ d $-dimensional Lebesgue measure restricted on the $ d $-dimensional unit cube $ [0,1]^d. $
	\begin{corollary}\label{MainCorollary}
		Assume Setting \ref{Setting31},
		let $ c,T\in(0,\infty) $, 
		let $ \nrm{\cdot}_F\colon(\cup_{d\in\N}\R^{d\times d})$  $\rightarrow [0,\infty) $ satisfy for all $d\in\N, A= (a_{i,j})_{i,j\in\{1,...,d\}}\in\R^{d\times d} $  that $ \nrm{A}_F= \sqrt{\sum_{i,j=1}^{d}(a_{i,j})^2}$,
		let $ \langle \cdot,\cdot\rangle\colon(\cup_{d \in \N}\R^d\times\R^d)\rightarrow\R $ satisfy for all $ d\in \N, $ $ x=(x_1,...,x_d), $  $y=(y_1,...,y_d)\in\R^d $ that $ \langle x,y\rangle = \sum_{i=1}^d x_iy_i, $ 
		let $ f\in C(\R,\R) $, for all $ d\in\N $ let $ u_d \in C^{1,2}([0,T]\times\R^d,\R) $, $ \mu_d=(\mu_{d,i})_{i\in\{1,...,d\}}\in C(\R^d,\R^d) $, $ \sigma_d = (\sigma_{d,i,j})_{i,j\in\{1,...,d\}}\in C(\R^d,\R^{d\times d}) $ satisfy for all $ t\in[0,T], x=(x_1,...,x_d),y =(y_1,...,y_d)\in \R^d $ that $ u_d(T,x)=g_d(x), $
		\begin{equation}
			|f(x_1)-f(y_1)|\le c|x_1-y_1|, \quad |u_d(t,x)|^2\le c\left[d^c+\nrm{x}^2\right],
		\end{equation}
		\begin{equation}
			\dfrac{\partial }{\partial t} u_d (t,x) +\langle(\nabla_x u_d)(t,x), \mu_d(x)\rangle+\tfrac{1}{2}\textup{Tr}\left(\sigma_d(x)[\sigma_d(x)]^*(\textup{Hess}_xu_d)(t,x)\right)+f(u_d(t,x))=0,
		\end{equation}
		for all $ d\in\N,v\in\R^d,\varepsilon\in(0,1] $ let $ \hat{\sigma}_{d,\varepsilon}\in C(\R^d,\R^{d\times d}), $ $ \frakg_{d,\varepsilon}, \tmu_{d,\varepsilon},\tsigma_{d,\varepsilon,v}\in\mathbf{N}  $ satisfying for all $d\in\N, x,y,v\in\R^d, \varepsilon\in(0,1], $ that $ \calR(\frakg_{d,\varepsilon})\in C(\R^d,\R) $, $\calR(\tmu_{d,\varepsilon})\in C(\R^d,\R^d)  $, $ \calR(\tsigma_{d,\varepsilon}) \in C(\R^d,\R^{d^2}) $, 
		$ \calR(\tsigma_{d,\varepsilon,v})=\hat{\sigma}_{d,\varepsilon}(\cdot)v,  \calD(\tsigma_{d,\varepsilon,v})=\calD(\tsigma_{d,\varepsilon,0}),  $ 
		\begin{equation}\label{COR4.2_g_mu_sigma_growth}
			\left |(\calR(\frakg_{d,\varepsilon}))(x)\right |^2+\max_{i,j\in\{1,...,d\}}\left(\left |[(\calR(\tmu_{d,\varepsilon}))(0)]_i\right |+\left |(\hat{\sigma}_{d,\varepsilon}(0))_{i,j}\right |\right)\le c\left[d^c+\nrm{x}^2\right],
		\end{equation}
		\begin{equation}\label{COR4.2_g_mu_sigma_approx}
			\max\{|g_d(x)-(\calR(\frakg_{d,\varepsilon}))(x)|, \nrm{\mu_d(x)-(\calR(\tmu_{d,\varepsilon}))(x)}, \nrm{\sigma_d(x)-\hat{\sigma}_{d,\varepsilon}(x)}_F\}\le \varepsilon cd^c(1+\nrm{x}^c),
		\end{equation}
		\begin{equation}\label{COR4.2_g_mu_sigma_lipschitz}
			\max\{|(\calR(\frakg_{d,\varepsilon}))(x)-(\calR(\frakg_{d,\varepsilon}))(y)|,\nrm{(\calR(\tmu_{d,\varepsilon}))(x)-(\calR(\tmu_{d,\varepsilon}))(y)},\nrm{\hat{\sigma}_{d,\varepsilon}(x)-\hat{\sigma}_{d,\varepsilon}(y)}_F \}\le c\nrm{x-y},
		\end{equation}
		\begin{equation}\label{COR4.2_DNN_PARAMS}
			\max\{\calP(\frakg_{d,\varepsilon}),\calP(\tmu_{d,\varepsilon}),\calP(\tsigma_{d,\varepsilon,v})\}\le cd^c\varepsilon^{-c}.
		\end{equation}
		Then there exist $ (\Psi_{d,\varepsilon})_{d\in\N, \varepsilon\in(0,1]} \subseteq\mathbf{N}, \eta\in(0,\infty)$ such that for all $ d\in\N, \varepsilon\in(0,1] $ it holds that $ \calR(\Psi_{d,\varepsilon})\in C(\R^d, \R) $, $ \calP(\Psi_{d,\varepsilon})\le \eta d^\eta\varepsilon^{-\eta} $ and 
		\begin{equation}
			\left[\int_{[0,1]^d}|u_d(0,x)-(\calR(\Psi_{d,\varepsilon}))(x)|^2dx\right]^{\frac{1}{2}}\le \varepsilon.
		\end{equation}
	\end{corollary}
	\begin{proof}
		Throughout the proof assume without loss of generality that $ c\in[2,\infty) $ (otherwise adapt the application of Theorem \ref{MainTheorem} below accordingly), 
		\color{black} and let $ B= 2c$. 
		We thus have that 
		\begin{equation}\label{COR4.2_proof_1}
			\max\{\mnrm{\calD(\frakg_{d,\varepsilon})},\dim(\calD(\frakg_{d,\varepsilon}))\}\le cd^c\varepsilon^{-c} \le Bd^c\varepsilon^{-c}.
		\end{equation}
		For all $ y\in [0,1]^d $ we have $ \nrm{y}\le \sqrt{d} $ and thus, for all $ \alpha \in (0,\infty) $ it holds that
		\begin{equation}\label{COR4.2_proof_3}
			\left(\int_{[0,1]^d}\nrm{y}^\alpha dy\right)^{\frac{1}{\alpha}}\le \sqrt{d} \le Bd.
		\end{equation}
		Jensen's inequality ensures for all $ d\in\N, x\in \R^d, \varepsilon\in (0,1] $ that 
		\begin{equation}\label{COR4.2_proof_4}
			\begin{aligned}
				|u_d(T,x)-(\calR(\frakg_{d,\varepsilon}))(x)|&\le \varepsilon cd^c(1+\nrm{x}^c) \\
				&\le \varepsilon cd^c 2(d^{2c}+\nrm{x}^2)^c\\
				&= \varepsilon B d^c(d^{2c}+\nrm{x}^2)^{c}.
			\end{aligned}
		\end{equation}
		Next,  (\ref{COR4.2_g_mu_sigma_growth})  ensures for all $ d\in \N , x\in \R^d,\varepsilon\in(0,1]$ that 
		\begin{equation}\label{COR4.2_proof_5}
			|(\calR(\frakg_{d,\varepsilon}))(x)|\le \sqrt{c}\left (d^c+\nrm{x}^2\right )^{\frac{1}{2}}\le (\sqrt{c}+c\sqrt{T})\left(d^{2c}+\nrm{x}^2\right)^{\frac{1}{2}},
		\end{equation}
		\begin{equation}\label{COR4.2_proof_7}
			\nrm{(\calR(\tmu_{d,\varepsilon}))(0)}\le \left(d \max_{i\in\{1,...,d\}}|[(\calR(\tmu_{d,\varepsilon}))(0)]_i|^2\right)^{\frac{1}{2}}\le cd^{c+1},
		\end{equation}
		and that 
		\begin{equation}\label{COR4.2_proof_8}
			\nrm{\hat{\sigma}_{d,\varepsilon}(0)}_F\le \left(d^2\max_{i,j\in\{1,...,d\}}|(\hat{\sigma}_{d,\varepsilon}(0))_{i,j}|^2\right)^{\frac{1}{2}}\le cd^{c+1}.
		\end{equation}
		Display (\ref{COR4.2_g_mu_sigma_lipschitz}) ensures for all $ d\in\N, x\in \R^d, \varepsilon\in (0,1] $ that 
		\begin{equation}\label{COR4.2_proof_6}
			|(\calR(\frakg_{d,\varepsilon}))(x)-(\calR(\frakg_{d,\varepsilon}))(y)|\le c\nrm{x-y}\le \left(\sqrt{\frac{c}{T}}+c\right)\nrm{x-y}.
		\end{equation}
		Furthermore, the fact that for all $d\in\N, t\in[0,T], x\in\R^d $ it holds that  $ |u_d(t,x)|^2\le c(d^c+\nrm{x}^2) $, Lipschitz continuity of $ f $, and the fact that every classical solution is also a viscosity solution ensure for all $ d\in\N $ that $ u_d $ satisfies (\ref{THM4.2_EXVISC}). This combined  with (\ref{COR4.2_proof_1})-(\ref{COR4.2_proof_6}), and 
		Theorem \ref{MainTheorem} (applied with $ \alpha \curvearrowleft c,$  $ \beta\curvearrowleft c, $ $b\curvearrowleft \sqrt{c}+c\sqrt{T}  $,  $B\curvearrowleft B,$  $ c\curvearrowleft c+1,$  $ p\curvearrowleft c, $  $q\curvearrowleft c ,$  $\mathfrak{p}\curvearrowleft 1, \gamma\curvearrowleft \frac{1}{2}$  in the notation of Theorem \ref{MainTheorem})  proves  existence of $ \left (\Psi_{d,\varepsilon}\right )_{d\in\N, \varepsilon\in (0,1]}\subseteq \mathbf{N}, \eta \in (0,\infty)  $ such that for all $ d\in\N, \varepsilon\in (0,1] $ it holds that $ \calR(\Psi_{d,\varepsilon}) \in C(\R^d,\R)$, $ \calP(\Psi_{d,\varepsilon})\le \eta d^{\eta}\varepsilon^{-\eta} $ and 
		\begin{equation}
			\left[\int_{[0,1]^d}|u_d(0,x)-(\calR(\Psi_{d,\varepsilon}))(x)|^2dx\right]^{\frac{1}{2}}\le \varepsilon.
		\end{equation}
		The proof is thus completed.
	\end{proof}
	\noindent \textbf{Acknowledgement.} The second author acknowledges funding through the research grant HU1889/7-1.
	\bibliography{refrences}
	\bibliographystyle{hacm}
\end{document}